\documentclass[11pt,british,reqno]{amsart}

\usepackage[T1]{fontenc}

\usepackage{babel,eucal,url,amssymb,stmaryrd,booktabs,enumerate,amscd,paralist,cite,color}

\theoremstyle{plain}
\newtheorem{lemma}{Lemma}[section]

\newtheorem{proposition}[lemma]{Proposition}
\newtheorem{corollary}[lemma]{Corollary}
\newtheorem{theorem}[lemma]{Theorem}
\newtheorem{remark}[lemma]{Remark}
\newtheorem{example}[lemma]{Example}

\newcommand{\Lie}[1]{\operatorname{\textsl{#1}}}
\newcommand{\lie}[1]{\operatorname{\mathfrak{#1}}}

\newcommand{\un}{\lie{u}}
\newcommand{\Gtwo}{\ifmmode{{\rm G}_2}\else{${\rm G}_2$}\fi}

 \newcommand{\cyclic}{\mathop{\kern0.9ex{{+}\kern-2.2ex\raise-.28ex\hbox{\Large\hbox
 {$\circlearrowright$}}}}}

\newcommand{\real}[1]{\left\llbracket #1 \right\rrbracket}

\newcommand{\lcf}{\lbrack\!\lbrack}
\newcommand{\rcf}{\rbrack\!\rbrack}

\def\sideremark#1{\ifvmode\leavevmode\fi\vadjust{\vbox to0pt{\vss
 \hbox to 0pt{\hskip\hsize\hskip1em
 \vbox{\hsize2.5cm\tiny\raggedright\pretolerance10000
 \noindent #1\hfill}\hss}\vbox to8pt{\vfil}\vss}}}%

\newfont{\eusm}{eusm10 scaled \magstep1}
\newfont{\eusmiii}{eusm10 scaled \magstep3}

\title{On the classification of almost contact metric manifolds}

\author{Francisco~Mart\'\i n~Cabrera}
\address[Francisco~Mart\'\i n~Cabrera]{Departamento de Matemáticas, Estadística e Investigación   Operativa \\
  University of La Laguna\\ 38200 La Laguna, Tenerife, Spain}
\email{fmartin@ull.edu.es}

\date{\today}
\begin{document}

\maketitle

\begin{abstract}
On  connected  manifolds of dimension higher than three,   the non-existence of $132$ Chinea and González-Dávila types of almost contact metric structures is proved.  This is a consequence of    some interrelations among  components   of the intrinsic torsion of an almost contact metric structure. Such interrelations allow to describe the exterior derivatives of some relevant forms in the context of almost contact metric geometry. 
\end{abstract}
\vspace{4mm}

{\footnotesize


}

\vspace{4mm}

\section{Introduction} 

In \cite{ChineaGonzalezDavila} Chinea and González-Dávila displayed a Gray-Hervella type classification for almost contact metric structures. Such a classification is based on the  decomposition of the space   possible intrinsic torsions into irreducible $\Lie{U}(n)$-modules ($\Lie{U}(n) \times 1$ is the structural group in case of almost contact metric structure). Since in general dimensions they obtained   a decomposition of the intrinsic torsion $\xi$ of the structure  into twelve $\Lie{U}(n)$-components $\xi_{(i)}$ respectively belonging to irreducible $\Lie{U}(n)$-modules $\mathcal{C}_i$,  from algebraic point view,  there are potentially $2^{12}$ classes. However, because of geometry, some of these classes could not exist on  connected manifolds. For instance, in \cite{JCMAR} Marrero  has  proved the non-existence of almost contact metric structure  of strict type $\mathcal{C}_5 \oplus \mathcal{C}_6$ defined on a connected manifold of dimension higher than $3$. A similar fact of non-existence has been  proved for 
$\Lie{G}_2$-structures in  \cite{Martin1} and for $\Lie{SU}(3)$-structures 
on six dimensional manifolds in \cite{Martin2}.  Each one of these results shows the non-existence of only one type of the considered $\Lie{G}$-structure.  In the present paper we prove the non-existence of $132$ types of almost contact metric structures defined on a connected manifold of dimension higher than $3$ (see Remark \ref{nonexist}). This is a consequence of   Theorem \ref{mainmain} below. Concretely we prove 
\vspace{1mm}

\textit{For an almost contact metric connected manifold of dimension $2n+1$, $n>1$:}

\noindent \textit{$(i)$  If the  structure is of type $\mathcal{C}_1 \oplus \mathcal{C}_2 \oplus \mathcal{C}_3 \oplus  \mathcal{C}_5 \oplus \mathcal{C}_6 \oplus \mathcal{C}_8 \oplus \mathcal{C}_9\oplus \mathcal{C}_{11}$ with $(\xi_{(5)} , \xi_{(6)}) \neq (0,0)$, then it is of type  $\mathcal{C}_1 \oplus \mathcal{C}_2 \oplus \mathcal{C}_3  \oplus \mathcal{C}_5 \oplus \mathcal{C}_8 \oplus \mathcal{C}_9 \oplus \mathcal{C}_{11}$, or of type 
$  \mathcal{C}_2 \oplus   \mathcal{C}_6    \oplus \mathcal{C}_9 $. } 
\vspace{1mm}

\noindent \textit{$(ii)$ If the  structure is of type $\mathcal{C}_2 \oplus \mathcal{C}_5 \oplus \mathcal{C}_7 \oplus  \mathcal{C}_9 $ with $(\xi_{(5)}, \xi_{(7)} ) \neq (0,0)$,  then it is of type  $\mathcal{C}_2 \oplus \mathcal{C}_7 \oplus \mathcal{C}_9$  or of type 
$  \mathcal{C}_2 \oplus   \mathcal{C}_5   \oplus \mathcal{C}_9  $. }

\vspace{1mm}

\noindent In the proof of these results we make use of some interrelations among components of the intrinsic torsion which are consequences of the identities  $d^2 \eta =0$ and $d^2F=0$, where $\eta$ is the one-form metrically equivalent to  the Reeb vector field $\zeta$  and  $F$ is the fundamental two-form  of the structure. Such interrelations are interesting on their own and give rise to  expressions for the exterior derivatives of the functions $d^* \eta$, $d^*F(\zeta)$ and the one-forms  $\xi_\zeta \eta$ and $\sum_{i=1}^{2n+1} (\xi_{(4) e_i} e_i)^\flat = \tfrac{n-1}{2} \theta$, where $\theta$ denotes the \textit{Lee form} considered in \cite{Gray-H:16},  $\{e_1, \dots , e_{2n+1} \}$ is an orthonormal basis for vectors and $X^\flat$ is the one-form metrically equivalent to the vector $X$. These functions and one-forms   determine the components of the intrinsic torsion in $\mathcal{C}_5$, $\mathcal{C}_6$,  $\mathcal{C}_{12}$ and $\mathcal{C}_{4}$, respectively. 

  Finally, we describe how to use the exterior derivatives $d\eta$, $d F$ and the Nijenhuis tensor $N_\varphi$ to determine the type of almost contact metric structure. This is used in some    examples.

\section{Preliminaries}{\indent}
\label{sect:prelima} \setcounter{equation}{0}
An \textit{ almost contact} structure $(\varphi, \zeta , \eta)$ on a manifold $M$ consists of a $(1,1)$-tensor $\varphi$, a vector field $\zeta$, called the {\it Reeb vector field},  and a one-form $\eta$ such that
\begin{equation} \label{blair}
\varphi^{2} = -I + \eta\otimes \zeta, \qquad \eta(\zeta) =1. 
\end{equation} 
Many authors include also that $\varphi \zeta=0$ and $\eta \circ \varphi =0$. However, such equalities are deducible  from \eqref{blair}. An explicit proof for this is given by Blair in \cite{Bl}.
The dimension of  $M$ must be $2n+1$. The presence of an almost contact structure is equivalent to say that there is a $\Lie{GL}(n, \mathbb{C}) \times 1$-structure defined  on $M$. 
 A manifold $M$ is said to be equipped with   an  \textit{almost contact
metric} structure, if there is an almost contact structure and a   Riemannian metric 
$\langle \cdot, \cdot \rangle$  on $M$ such that the following compatibility condition  is satisfied  
\begin{equation} \label{compat}
 \langle \varphi
X,\varphi Y\rangle = \langle X,Y\rangle - \eta(X)\eta(Y).
\end{equation}
This is   equivalent to say that there is a $\Lie{U}(n) \times 1$-structure defined on $M$.  We will make reiterated use of the \textit{musical isomorphisms} $\flat  :  \mathrm{T} M \to  \mathrm{T}^* M$
 and $\sharp : \mathrm{T}^* M \to \mathrm{T} M$, induced by  $\langle \cdot , \cdot \rangle$, defined respectively by $X^\flat = \langle X, \cdot  \rangle$  and $\langle \theta^\sharp , \cdot \rangle  = \theta$. Thus, using $\varphi \zeta =0$ and \eqref{compat}, one has  $\eta = \zeta^\flat$ and $\zeta = 	\eta^\sharp$.

Associated with an  almost contact 
metric structure,   the tensor $F= \langle \cdot, \varphi
\cdot\rangle$, called the {\em fundamental two-form}, is usually
considered. Using $F$ and $\eta$, $M$ can be oriented by fixing a
constant multiple of $F^n \wedge \eta = F \wedge
\stackrel{(n)}{\dots}\wedge F \wedge \eta$ as volume form. 

For almost contact metric structures, the cotangent space on each point
$\mbox{T}^{*}_m M$ is not irreducible under the action of the group
$\Lie{U}(n) \times 1$. In fact, $\mbox{T}^* M = \eta^{\perp} \oplus
\mathbb{R} \eta$,  where $\eta^\perp$ is the image under by $\flat$ of the distribution $\zeta^\perp$ orthogonal to $\zeta$. Taking this into account, it follows 
$$
\lie{so}(2n+1) \cong \Lambda^{2} \mbox{T}^* M = \Lambda^2
\eta^{\perp} \oplus \eta^{\perp} \wedge \mathbb R \eta.
$$
From now on we will denote $X_{\zeta^{\perp}} = X - \eta(X) \zeta$,
for all vector field  $X$. Since $\Lambda^2 \eta^{\perp} =
\lie{u}(n) \oplus \lie{u}(n)^{\perp}_{|\zeta^{\perp}} $, where
$\lie{u}(n)$ ($\lie{u}(n)^{\perp}_{|\zeta^{\perp}} $)
consists of those two-forms $b$ such that $b(\varphi X , \varphi Y)
= b( X_{\zeta^{\perp}} ,  Y_{\zeta^{\perp}})$ ($b(\varphi X ,
\varphi Y) = - b( X_{\zeta^{\perp}} , Y_{\zeta^{\perp}})$), we have
$$
\lie{so}(2n+1) = \lie{u}(n) \oplus \lie{u}(n)^{\perp}, \quad
\mbox{with } \lie{u}(n)^{\perp} =\lie{u}(n)^{\perp}_{|\zeta^{\perp}}
 \oplus \eta^{\perp} \wedge
\mathbb R \eta.
$$

 Denoting by $\nabla$ the Levi Civita connection, the {\it minimal
connection} $\nabla^{\Lie{U}(n)}$  is  the unique $\Lie{U}(n)$-connection  such that  $\xi_X = \nabla^{\Lie{U}(n)} _X-\nabla_X$  satisfies the condition $\xi_X \in  \lie{u}(n)^{\perp}$. The tensor  $\xi$ is referred to as the {\it intrinsic
torsion} of the almost contact metric structure 
\cite{CleytonSwann:torsion}.  The space $\mbox{T}^* M \otimes \lie{u}(n)^{\perp}$
of intrinsic torsions   has the following first decomposition into $\Lie{U}(n)$-modules:
$$
\mbox{T}^* M \otimes  \lie{u}(n)^{\perp} = (\eta^{\perp} \otimes
\lie{u}(n)^{\perp}_{|\zeta^{\perp}}) \oplus (\eta \otimes
\lie{u}(n)^{\perp}_{|\zeta^{\perp}}) \oplus ( \eta^{\perp} \otimes
\eta^{\perp} \wedge   \eta  ) \oplus (\eta \otimes
  \eta^{\perp} \wedge   \eta).
$$
Chinea  and Gonz{\'a}lez-D{\'a}vila \cite{ChineaGonzalezDavila}
showed that $\mbox{T}^* M \otimes \lie{u}(n)^{\perp}$ is
decomposed into twelve irreducible $\Lie{U}(n)$-modules $\mathcal
C_1, \dots , \mathcal C_{12}$, where
$$
\begin{array}{rclrcl}
\eta^{\perp} \otimes \lie{u}(n)^{\perp}_{|\zeta^{\perp}}
 & = &
\mathcal
C_1 \oplus \mathcal C_2 \oplus \mathcal C_3 \oplus \mathcal C_4, 
& \quad \eta^{\perp} \otimes \eta^{\perp} \wedge  \eta  & = & \mathcal C_5
\oplus \mathcal C_8 \oplus \mathcal C_9 \oplus \mathcal C_6 \oplus
\mathcal C_7  \oplus \mathcal C_{10}, \\
 \eta \otimes \lie{u}(n)^{\perp}_{|\zeta^{\perp}}
& = & \mathcal C_{11} , &
 \eta \otimes \eta^{\perp} \wedge   \eta & = & \mathcal
C_{12}.
\end{array}
$$
The $\Lie{U}(n)$-modules  $\mathcal C_1, \dots , \mathcal C_4$ are isomorphic
to  the Gray and Hervella's ones given  in \cite{Gray-H:16}. Furthermore, note that $\varphi$
restricted to $\zeta^{\perp}$ works as an almost complex structure
and, if one considers the $\Lie{U}(n)$-action on the bilinear
forms $\otimes^2 \eta^{\perp}$, then one has the decomposition
$$
\textstyle \otimes^2 \eta^{\perp} = \mathbb R \langle \cdot ,
\cdot \rangle_{|\zeta^{\perp}} \oplus \lie{su}(n)_s \oplus
\real{\sigma^{2,0}} \oplus \mathbb R F  \oplus \lie{su}(n)_a
\oplus \lie{u}(n)^{\perp}_{|\zeta^{\perp}}.
$$
The modules  $\lie{su}(n)_s$ ($\lie{su}(n)_a$) consists of
Hermitian symmetric (skew-symmetric) bilinear forms
orthogonal to $\langle \cdot , \cdot \rangle_{|\zeta^{\perp}}$
($F$),
 and  $\real{\sigma^{2,0}}$ ($\lie{u}(n)^{\perp}_{|\zeta^{\perp}}
 $) is the space of  anti-Hermitian
symmetric (skew-symmetric) bilinear forms. With respect to
the modules $\mathcal C_i$, one has $\eta^{\perp} \otimes
\eta^{\perp} \wedge  \mathbb R \eta \cong   \otimes^2
\eta^{\perp}$ and, using the $\Lie{U}(n)$-map $\xi^{\Lie{U}(n)}
\to - \xi^{\Lie{U}(n)} \eta = \nabla \eta$, it is obtained
$$
\mathcal C_5 \cong  \mathbb R \langle \cdot , \cdot
\rangle_{|\zeta^{\perp}} , \quad \mathcal C_8 \cong \lie{su}(n)_s,
\quad \mathcal C_9 \cong \real{\sigma^{2,0}}, \quad \mathcal C_6
\cong \mathbb R F, \quad \mathcal C_7 \cong  \lie{su}(n)_a, \quad
\mathcal C_{10} \cong \lie{u}(n)^{\perp}_{|\zeta^{\perp}}.
$$
In summary,  the space of intrinsic torsions
  $\mbox{T}^* M\otimes  \lie{u}(n)^{\perp}$ consists of those tensors
$\xi$ such that
\begin{equation} \label{inttorcar}
  \varphi \xi_X Y + \xi_X \varphi Y =
  \eta (Y) \varphi \xi_X \zeta +
\eta ( \xi_X \varphi Y) \zeta
\end{equation}
and, under the action of $U(n)\times 1$, is decomposed into:
\begin{enumerate}
\item[1.] if $n=1$, $ \xi \in \mbox{T}^* M \otimes
\un(1)^\perp = \mathcal C_{5} \oplus \mathcal C_{6} \oplus
\mathcal C_{9} \oplus \mathcal C_{12}$; 
\item[2.] if $n=2$, $
\xi \in \mbox{T}^* M \otimes \un(2)^\perp = \mathcal
C_{2} \oplus \mathcal C_{4} \oplus \dots \oplus \mathcal C_{12}$;
\item[3.] if $n \geqslant 3$, $ \xi \in \mbox{T}^* M
\otimes \un(n)^\perp =
  \mathcal C_{1}  \oplus \dots  \oplus \mathcal C_{12}$.
\end{enumerate}

  Some of these classes are referred to,  by
diverse authors \cite{Bl,ChineaGonzalezDavila}, as:

 $\{ \xi =0 \}=$ cosymplectic manifolds or integrable almost contact metric structure,
 $\;\mathcal C_1=$ nearly-K-cosymplectic manifolds,
 $\;\mathcal C_5=$ $\alpha$-Kenmotsu manifolds,
 $\;\mathcal C_6=$ $\alpha$-Sasakian manifolds,
  $\;\mathcal C_5\oplus \mathcal C_6=$ trans-Sasakian manifolds,
 $\;\mathcal C_2 \oplus \mathcal C_9=$ almost cosymplectic manifolds,
  $\;\mathcal C_6 \oplus \mathcal C_7=$ quasi-Sasakian manifolds,
 $\;\mathcal C_1 \oplus \mathcal C_5 \oplus \mathcal C_6=$ nearly-trans-Sasakian manifolds,
 $\;\mathcal C_1 \oplus \mathcal C_2 \oplus \mathcal C_9 \oplus \mathcal C_{10}=$ quasi-K-cosymplectic manifolds,
 $\;\mathcal C_3 \oplus \mathcal C_4 \oplus \mathcal C_5 \oplus \mathcal C_{6}
  \oplus \mathcal C_7 \oplus \mathcal C_{8}=$ normal manifolds, 
   $\;\mathcal C_3 \oplus \mathcal C_4 \oplus \mathcal C_5  \oplus \mathcal C_{8} =$ integrable almost contact structure, 
   etc.
  \vspace{2mm}

The intrinsic  torsion is given    by
\begin{align*} 
  \xi_X   =   -  \tfrac12  \varphi \circ \nabla_X \varphi +
\nabla_X \eta \otimes \zeta - \tfrac12 \eta \otimes  \nabla_X \zeta 
 =  \tfrac12 (\nabla_X \varphi) \circ \varphi   + \tfrac12
\nabla_X \eta \otimes \zeta - \eta \otimes \nabla_X \zeta
\nonumber
\end{align*}
(see \cite{GDMC1}). If the almost contact metric structure is
of type $\mathcal{C}_5 \oplus \dots \oplus \mathcal{C}_{10} \oplus
\mathcal{C}_{12}$, then the expression for the intrinsic torsion is
reduced to
 $ \xi_X   =  \nabla_X \eta \otimes \zeta - \eta \otimes \nabla_X
  \zeta$.

 The tensor $\xi_{(i)}$ will denote the component of
$\xi$  obtained by the $\Lie{U}(n)$-isomorphism  $(\nabla F)_{(i)}
=(-\xi F)_{(i)}
 \in \mathcal C_i \to \xi_{(i)}$. In this way we
are using the same terminology  used in
\cite{ChineaGonzalezDavila} by Chinea and González-Dávila when we
are referring to classes.

Some vector fields are involved in the characterization of certain types of almost contact metric manifolds. For instance,  if  $d^*$ denotes the coderivative and  $\{e_1, \dots , e_{2n+1} \}$ is a local  orthonormal frame field, 
the vector field $ \sum_{i=1}^{2n+1}   \xi_{e_i} e_i $ is given by 
\begin{equation} \label{firstvector}
  \textstyle   \sum_{i=1}^{2n+1}   \xi_{e_i} e_i  = - \tfrac12  \varphi (d^* F)^{\sharp} - d^* \eta \;
\zeta - \tfrac12  \nabla_{\zeta}\zeta.
\end{equation}
This identity follows from $\nabla F = - \xi F$ and using \eqref{inttorcar}. Now from the properties of each component $\xi_{(j)}$ of $\xi$, we have  $\sum_{i=1}^{2n+1}  \xi_{{(j)}e_i} e_i \neq 0 $ only for $j=4,5,12$ (see \cite{ChineaGonzalezDavila}, \cite{ChineaJCMarr1}), i.e. our vector field is only contributed  by the components of $\xi$ in 
$\mathcal C_4$, $\mathcal C_5$ and $\mathcal C_{12}$. 
 Then
\begin{equation} \label{sum4512}
 \textstyle \sum_{i=1}^{2n+1}   \xi_{e_i} e_i  = \sum_{i=1}^{2n+1}  \xi_{{(4)}e_i} e_i + \sum_{i=1}^{2n+1}  \xi_{{(5)}e_i} e_i + \sum_{i=1}^{2n+1}  \xi_{{(12)}e_i} e_i.
\end{equation}
On the other hand, it easy to see that  
$$ 
\textstyle   \sum_{i=1}^{2n+1}  \xi_{{(5)}e_i} e_i = - d^* \eta \; \zeta,  \qquad
 \sum_{i=1}^{2n+1}  \xi_{{(12)}e_i} e_i = -
 \nabla_{\zeta}\zeta.
 $$
Finally, from these two identities, \eqref{firstvector} and \eqref{sum4512}, it is deduced
\begin{equation} \label{leeac}
\textstyle  \sum_{i=1}^{2n+1}  \xi_{{(4)}e_i} e_i = - \tfrac12  \varphi (d^* F)^{\sharp} + \tfrac12 
\nabla_{\zeta}\zeta. 
 \end{equation}
\vspace{2mm}

Moreover, one has 
the vector field  $ \textstyle \sum_{i=1}^{2n+1} \xi_{e_i} \varphi e_i $. By using similar arguments as before we have the identity
$$
 \textstyle \sum_{i=1}^{2n+1} \xi_{e_i} \varphi e_i = - \tfrac12  (d^* F)^\sharp  - \tfrac12  d^* F(\zeta) \zeta
- \varphi \nabla_\zeta \zeta.
$$
 This second  vector field is only contributed by the components of $\xi$ in
$\mathcal C_4$ and   $\mathcal C_6$. It follows that such contributions are given by   
$$ \textstyle   \sum_{i=1}^{2n+1}   \xi_{{(4)}e_i}
\varphi e_i = - \tfrac12 (d^* F)^{\sharp} -  \varphi \nabla_\zeta \zeta + \tfrac12 d^*
F(\zeta) \zeta, \quad   \quad \sum_{i=1}^{2n+1}  \xi_{{(6)}e_i} \varphi e_i = - d^* F(\zeta)
\zeta.
$$ 
\vspace{2mm}

For a $2n$-dimensional  almost Hermitian manifold $(M,J,\langle \cdot, \cdot \rangle)$, where $J$ is the almost complex structure and  $\langle \cdot, \cdot \rangle$ is the metric compatible with $J$, the \textit{Lee one-form} $\theta$  is defined by $\theta = - \frac{1}{n-1} J d^* \omega$, where $\omega = \langle \cdot , J \cdot \rangle$ is the Kähler two-form (see \cite{Gray-H:16}). The one-form $\theta$ determines the component usually denoted by  $\xi_{(4)}$ of the intrinsic torsion of the almost Hermitian structure. Such a component is given by
$$
4 \xi_{(4)X} = X^\flat \otimes \theta^\sharp - \theta \otimes X-  JX^\flat \otimes J\theta^\sharp +  J \theta  \otimes J X.
$$
 Note that $\sum_{i=1}^{2n} \xi_{e_i} e_i = \frac{n-1}{2} \theta^\sharp$.
 \vspace{2mm}
 
 In  the context of almost contact metric geometric, taking \eqref{leeac} into account,   the Lee form is defined by 
 $$
(n-1) \theta  =  -   \varphi (d^* F)^{\sharp} +
\nabla_{\zeta}\eta, 
$$
where $2n+1$ is the dimension of the almost contact metric  manifold $(M, \varphi, \zeta, \eta, \langle \cdot , \cdot \rangle)$. 
The  component $\xi_{(4)}$ is given by
$$
4 \xi_{(4)X} = X_{\zeta^\perp}^\flat \otimes \theta^\sharp - \theta \otimes X_{\zeta^\perp}-  \varphi X^\flat \otimes \varphi \theta^\sharp +  \varphi \theta  \otimes \varphi X,
$$
where $X_{\zeta^\perp} = X - \eta(X) \zeta$.

\begin{remark}
  {\rm For using simpler and standard notation,  we recall that $\lambda_0^{p,q}$ is a complex irreducible $\Lie{U}(n)$-module coming from the $(p,q)$-part of the  
complex exterior algebra, and that its corresponding dominant weight in standard
coordinates is given by $(1, \dots,1,0, \dots,0, -1, \dots , -1)$, where $1$ and $-1$ are repeated $p$ and $q$ times, respectively. By analogy with the exterior algebra, there
are also complex irreducible $\Lie{U}(n)$-modules $\sigma^{p,q}_0$, with dominant weights $(p,0, \dots ,0, -q)$ 
coming from the complex symmetric algebra. The notation $\lcf V \rcf$ stands for the real vector space underlying a complex vector space $V$, and $[W ]$ denotes a real vector space that admits $W$ as its complexification. Thus for the $\Lie{U}(n)$-modules above mentioned one has
$$
\eta^{\perp} \cong \lcf \lambda^{1,0} \rcf, \quad  \lie{u}(n)  \cong  [\lambda^{1,1}], \quad  \lie{su}(n)_s \cong  \lie{su}(n)_a \cong  [\lambda^{1,1}_0] ,  \quad \lie{u}(n)^{\perp}_{|\zeta^{\perp}} \cong  \lcf \lambda^{2,0} \rcf.
$$

The space of two forms  $\Lambda^2 \mathrm{T}^* M$ is decomposed into irreducible  $\Lie{U}(n)$-components as follows:
$$
\Lambda^2 \mathrm{T}^* M =   \mathbb{R} \, F+ [\lambda_0^{1,1}]  + \lcf \lambda^{2,0} \rcf +   \eta \wedge  \lcf \lambda^{1,0} \rcf . 
$$
The components of a two-form $\alpha$ are given by 
\begin{gather*}
2 \alpha_{[\lambda^{1,1}]} (X,Y)  =  \alpha (\varphi^2 X , \varphi^2 Y) +  \alpha (\varphi X , \varphi Y),  \quad
2 \alpha_{\lcf \lambda^{2,0}\rcf} (X,Y)  =  \alpha (\varphi^2 X , \varphi^2 Y) -  \alpha (\varphi X , \varphi Y),  \\
\alpha_{\eta \wedge \lcf \lambda^{1,0} \rcf  }  =  \eta  \wedge (\zeta \lrcorner \alpha),
\end{gather*}
where $\lrcorner$ denotes the interior product.
We will use the natural extension  to forms of the metric $\langle \cdot , \cdot \rangle$. Thus, for all $p$-forms $\alpha$, $\beta$,  
 $$
 \langle \alpha , \beta \rangle = \tfrac{1}{p!} \textstyle  \sum_{i_1, \dots, i_p=1}^{2n+1} \alpha (e_{i_1} , \dots , e_{i_p})
\beta (e_{i_1} , \dots , e_{i_p}).
$$ 
 Using this product we have $\alpha_{\mathbb{R}F} = \frac1{n} \langle \alpha , F \rangle F$.
 
  In the sequel, we will consider the orthonormal basis for  vectors   $\{e_1, \dots, e_{2n}, e_{2n+1} = \zeta\}$. Likewise, we will use the summation convention. The repeated indexes will mean that the sum is extended from $i=1$ to $i=2n$.  Otherwise, the sum will be explicitly  written. 

   }
  \end{remark}

\section{Exterior derivatives of relevant forms of the structure} 
\setcounter{equation}{0}

In this section we will display  several identities relating components of the intrinsic torsion which are consequences of the equalities $d^2 F=0$ and $d^2 \eta=0$. They are interesting  on their own and we will use later some of them. 
 Some of those  identities   were already obtained in \cite{GDMC1} for the particular case of  almost contact metric structures of type $\mathcal{C}_1 \oplus \ldots \oplus \mathcal{C}_{10}$. Here the parts  $\xi_{(11)}$ and $\xi_{(12)}$ of the intrinsic torsion  are  also considered.  
 
 As applications of the  identities, we will obtain the $\Lie{U}(n)$-components of the exterior derivatives of  the one-forms  
$\theta$  and $\xi_\zeta \eta$, and the functions $d^*\eta$, $d^*F(\zeta)$. These one-forms and functions determine $\xi_{(4)}$,  
  $\xi_{(12)}$,  $\xi_{(5)}$ and  $\xi_{(6)}$, respectively.       
 
 \begin{lemma} \label{lambdaunouno}
  For almost contact metric manifolds of dimension $2n+1$, $n>1$, the
following identity is satisfied
{\rm  \small
 \begin{align*}
   0   = &
    \tfrac{n-2}{n-1} \textstyle   \langle \nabla^{\Lie{U}(n)}_{\varphi^2  X} \xi_{(4)e_i} { e_i},    Y \rangle
   -  \tfrac{n-2}{n-1} \textstyle   \langle \nabla^{\Lie{U}(n)}_{\varphi^2 Y} \xi_{(4) e_i} { e_i},    X \rangle
  - \tfrac{2}{n-1} \textstyle  (\nabla^{\Lie{U}(n)}_{e_j}  (\xi_{(4)e_i} { e_i})^\flat)) (\varphi e_j)  F(X,Y)
  \\
  &
   - 2 \textstyle  \langle (\nabla^{\Lie{U}(n)}_{e_i} \xi_{(3)})_{ X} Y,  { e_i} \rangle
    + 2 \textstyle  \langle (\nabla^{\Lie{U}(n)}_{e_i} \xi_{(3)})_{ Y}   X, { e_i} \rangle
  - \tfrac{n-2}{n-1} \textstyle  \langle \nabla^{\Lie{U}(n)}_{\varphi  X} \xi_{(4)e_i} { e_i}, \varphi  Y \rangle
   \\
   &
   +  \tfrac{n-2}{n-1} \textstyle   \langle \nabla^{\Lie{U}(n)}_{\varphi Y} \xi_{(4)e_i} { e_i}, \varphi  X \rangle
- 3 \textstyle  \langle \xi_{(1)X} e_i, \xi_{(2)Y} e_i \rangle 
+ 3  \textstyle  \langle \xi_{(1)Y} e_i, \xi_{(2)X} e_i \rangle     
\\
&
- \tfrac{2}{n^2} d^*\eta d^*F(\zeta) F(X,Y) 
+ \tfrac{4}{n} d^*\eta (\xi_{(7)X} \eta)(Y) 
- \tfrac{4}{n} d^*F(\zeta)    (\xi_{(8)X}  \eta) (\varphi Y) 
\\
&
+ 4 \langle   \xi_{(7)X}  \zeta, \xi_{(8)Y}  \zeta \rangle -  4 \langle   \xi_{(7)Y}  \zeta, \xi_{(8)X}  \zeta \rangle
+ 4 \langle   \xi_{(11)\zeta }  X , \xi_{(10)Y}  \zeta \rangle -  4 \langle   \xi_{(11)\zeta }  Y, \xi_{(10)X}  \zeta \rangle.
\end{align*} }
 \end{lemma}
 \begin{proof}
   The proof follows in a similar way as in \cite[Lemma 4.5, page 163]{GDMC1} for other identities below. Firstly  we note that
 $ d^2 F \in  \Lambda^4T^*M $ and  $\Lambda^4T^*M$  has the following $\Lie{U}(n)$-decomposition
 \begin{eqnarray*}
\Lambda^4T^*M & = & \lcf \lambda^{4,0} \rcf  \oplus \lcf \lambda^{3,1}\rcf \oplus \lcf \lambda^{2,0}\rcf  \wedge F \oplus
  [\lambda^{2,2}_0] \oplus [\lambda^{1,1}_0]\wedge  F \oplus \mathbb R F\wedge F
  \\
  && 
  \oplus \lcf \lambda^{3,0} \rcf \wedge \eta \oplus \lcf \lambda^{2,1}_0 \rcf
\wedge \eta
   \oplus \lcf \lambda^{1,0} \rcf\wedge  F \wedge \eta.
  \end{eqnarray*}
  Then   $d^2 F$ is written in terms of $\nabla^{\Lie{U}(n)}$ and $\xi$, i.e.
  \begin{align}
  d^2F(X_1,X_2,X_3,X_4)  = & \textstyle \sum_{1 \leq a < b \leq 4} (-1)^{a+b} \left( \left((\nabla^{\Lie{U}(n)}_{X_a} \xi)_{X_b} - (\nabla^{\Lie{U}(n)}_{X_b} \xi)_{X_a}\right) F\right) (X_c, X_d) \nonumber
   \\
& 
\textstyle + \sum_{1 \leq a < b \leq 4} (-1)^{a+b}(\xi_{\xi_{X_a} X_b - \xi_{X_b} X_a}  F) (X_c, X_d) \label{d2F}
\\
&
\textstyle - \sum_{1 \leq a < b \leq 4} (-1)^{a+b}([\xi_{X_a} ,\xi_{X_b}]  F) (X_c, X_d), \nonumber
  \end{align}
 where $c < d$, $ \{c,d\} =\{1, \dots,4\}-\{a,b\}$ in each case and $[\xi_{X_a} , \xi_{X_b} ] = \xi_{X_a}  \xi_{X_b} - \xi_{X_b}  \xi_{X_a}$.  Now contracting  with $F$ on the first two arguments, it is obtained the two-form
 \begin{equation} \label{usouso}
{\small  \begin{array}{rl} & \hspace{-.9cm}  \tfrac12 F_{1\,2} (d^2
F) (X,Y) =
  \\[1mm]
& - 2 \langle ( \nabla^{\Lie{U}(n)}_{e_i} \xi )_{\varphi e_i} X,
\varphi Y \rangle
 + 2 \langle ( \nabla^{\Lie{U}(n)}_{ \varphi e_i} \xi )_X Y,   { \varphi e_i} \rangle
  - 2 \langle ( \nabla^{\Lie{U}(n)}_{\varphi  e_i} \xi )_Y X, \varphi  e_i \rangle
  + 2 \langle ( \nabla^{\Lie{U}(n)}_X \xi )_{ \varphi e_i} {\varphi e_i},  Y \rangle
\\[1mm]
  &
  - 2 \langle ( \nabla^{\Lie{U}(n)}_Y \xi )_{\varphi e_i} { \varphi e_i},  X \rangle
  - 2  \langle \xi_{ \xi_{e_i} {\varphi e_i}} X , \varphi Y \rangle
  + 2 \langle \xi_X e_i  , \xi_{e_i} Y , \rangle
  - 2  \langle  \xi_Y { e_i} , \xi_{e_i} X  \rangle
\\[1mm]
    &
    - \eta(X) \left( (\nabla^{\Lie{U}(n)}_Y \xi )_{e_i} \eta \right) (e_i)
    + \eta(Y) \left( (\nabla^{\Lie{U}(n)}_X \xi )_{e_i} \eta \right) (e_i)
 - 2 \eta \odot \left( (\nabla^{\Lie{U}(n)}_{e_i} \xi )_{\varphi e_i} \eta \right) \circ \varphi (X,Y)
\\[1mm]
&
 +  \eta(X) \left( (\nabla^{\Lie{U}(n)}_{e_i} \xi )_{Y} \eta \right) (e_i)
 +  \eta(Y) \left( (\nabla^{\Lie{U}(n)}_{e_i} \xi )_{X} \eta \right) (e_i)
 +  2 \eta(X) (\xi_{\xi_{e_i} Y} \eta) (e_i)
 \\[1mm]
 &
 + 2 (\xi_{e_i} \eta)(X) (\xi_{\varphi e_i} \eta)\circ \varphi (Y)
 - 2 (\xi_{e_i} (\xi_Y \eta))  ({e_i}) \eta(X)
 + 2 (\xi_X \eta)(Y) (\xi_{ e_i} \eta)({e_i})
 \\[1mm]
 &
 + 2 (\xi_Y \eta)({e_i}) (\xi_{\varphi e_i} \eta)\circ \varphi (X)
 + 2 (\xi_{e_i} (\xi_{\varphi e_i} \eta)) \circ \varphi (X) \eta(Y)
 -  (\xi_{ e_i} \eta)({\varphi e_i}) (\xi_X \eta)\circ \varphi (Y)
 \\[1mm]
 &
 +  (\xi_{ e_i} \eta)({\varphi e_i}) (\xi_Y \eta)\circ \varphi (X)
   -  3 (\xi_{\xi_X \zeta} \eta)(Y)
 +  3 (\xi_{\xi_Y \zeta} \eta)(X).
\end{array} }
\end{equation}
Then  taking the
corresponding projection  to $\left[\lambda^{1,1}\right]$ of this two-form ,  it is computed the  two-form which determines the $[\lambda^{1,1}]$-component  of $d^2F$.   Finally using the properties of the components $\xi_{(i)}$ of~$\xi$ and the fact $d^2F=0$,  one has the required identity. We recall that $\nabla^{\Lie{U}(n)}$ is a $\Lie{U}(n)$-connection. This fact implies $\nabla^{\Lie{U}(n)} \xi_{(i)}$ is in $\mathcal{C}_i$, $\nabla^{\Lie{U}(n)}F=0$ and $\nabla^{\Lie{U}(n)} \eta =0$. Also from $\nabla^{\Lie{U}(n)} = \nabla + \xi$, one has $\nabla F = - \xi F$, and  $\nabla \eta = - \xi \eta$.
  \end{proof}

In  previous Lemma, if we use the equality
\begin{gather*}
\textstyle    (\nabla^{\Lie{U}(n)}_X (\xi_{(4)e_i} { e_i})^\flat) (Y) - (\nabla^{\Lie{U}(n)}_Y (\xi_{(4)e_i} { e_i})^\flat) (X)  = 
\textstyle  d  (\xi_{(4)e_i} { e_i})^\flat(X,Y) -  \langle \xi_X Y - \xi_Y X , \xi_{(4)e_i} { e_i} \rangle.    
\end{gather*}
we will obtain the $[\lambda^{1,1}]$-component  of the exterior derivative of the Lee form $\theta$.
 \begin{proposition} \label{divergenciaunouno}
 For almost contact metric manifolds of dimension $2n+1$, $n>1$,  we have
 \begin{align*}
   \tfrac{n-2}{2} \textstyle  d \theta_{[\lambda^{1,1}]} ( X,  Y)    = &
     \tfrac{1}{2} \textstyle  \langle d \theta , F \rangle   F(X,Y) 
   - \textstyle  \langle (\nabla^{\Lie{U}(n)}_{e_i} \xi_{(3)})_{ X} Y,  { e_i} \rangle
    + \textstyle   \langle (\nabla^{\Lie{U}(n)}_{e_i} \xi_{(3)})_{ Y}   X, { e_i} \rangle
    \\
 &
  +  \tfrac{n-2}{2} \textstyle  \theta (\xi_{(3)  X} Y- \xi_{(3)  Y} X) 
  - \tfrac{3}{2} \textstyle  \langle \xi_{(1)X} e_i, \xi_{(2)Y} e_i \rangle 
  + \tfrac{3}{2} \textstyle  \langle \xi_{(1)Y} e_i, \xi_{(2)X} e_i \rangle     
\\
&
- \tfrac{1}{n^2} d^*\eta d^*F(\zeta) F(X,Y) 
+ \tfrac{2}{n} d^*\eta (\xi_{(7)X} \eta)(Y) 
- \tfrac{2}{n} d^*F(\zeta)    (\xi_{(8)Y}  \eta) (\varphi X) 
\\
&
+ 2 \langle   \xi_{(7)X}  \zeta, \xi_{(8)Y}  \zeta \rangle -  2 \langle   \xi_{(7)Y}  \zeta, \xi_{(8)X}  \zeta \rangle
\\
&
+ 2 \langle   \xi_{(11)\zeta }  X , \xi_{(10)Y}  \zeta \rangle -  2 \langle   \xi_{(11)\zeta }  Y, \xi_{(10)X}  \zeta \rangle
\end{align*}
and
$
 \textstyle   (d\theta)_{\mathbb{R}} ( X,  Y)  = \tfrac{1}{n}  \textstyle   \langle d \theta , F \rangle F(X,Y), 
$
where
\begin{gather*}
 \textstyle
  \langle d \theta , F \rangle 
   =
\tfrac{1}{n}   d^*\eta d^*F(\zeta) - 2  \textstyle 
  \langle   \xi_{(7) \varphi e_i }  \zeta, \xi_{(8) e_i}  \zeta \rangle  
  -  2 \textstyle    \langle \xi_{(11)\zeta }  \varphi e_i  , \xi_{(10)e_i}  \zeta \rangle.
\end{gather*}
 \end{proposition}
   The identity in next Lemma is also a consequence of $d^2F=0$.  
\begin{lemma} 
 For almost contact metric manifolds of dimension$2n+1$, $n>1$, the
following identity is satisfied
\vspace{-1mm}
{\rm  \small
     \begin{align*}
     0= &
      3 \textstyle  \langle (\nabla^{\Lie{U}(n)}_{e_i} \xi_{(1)} )_{e_i} X_{}, Y_{} \rangle
   - \textstyle  \langle (\nabla^{\Lie{U}(n)}_{e_i} \xi_{(3)} )_{e_i} X_{}, Y_{} \rangle
   + (n-2) \textstyle   \langle (\nabla^{\Lie{U}(n)}_{e_i} \xi_{(4)} )_{e_i} X_{},Y_{}
   \rangle
  \\
  &
  - \textstyle   \langle \xi_{{(3)}X_{}} e_i,  \xi_{{(1)} Y_{}} e_i \rangle
  + \textstyle   \langle  \xi_{{(3)}Y_{}} e_i,  \xi_{{(1)}  X_{} }e_i \rangle
 +  \frac12\textstyle   \langle  \xi_{{(3)}X_{}} e_i,  \xi_{{(2)} Y_{}} e_i \rangle
   - \textstyle   \frac12  \langle  \xi_{{(3)}Y_{}} e_i,  \xi_{{(2)}  X_{}} e_i \rangle
    \\
   &
 \textstyle   - \tfrac{n-5}{n-1}  \textstyle  \langle \xi_{{(1)} \xi_{{(4)}e_i} e_i }X_{}, Y_{} \rangle
    - \tfrac{n-2}{n-1}  \textstyle  \langle  \xi_{{(2)} \xi_{{(4)}e_i} e_i } X_{}, Y_{} \rangle
      +  \textstyle   \langle \xi_{{(3)} \xi_{(4)e_i} e_i } X_{}, Y_{} \rangle
     \\
     &
     + \textstyle  (\xi_{(6)e_i} \eta) (\varphi e_i)  \langle  \xi_ {(11)\zeta}  X_{},  \varphi Y_{} \rangle
     + (n-2) \textstyle  (\xi_{(5)e_i} \eta ) \wedge   (  \xi_{(10) e_i} \eta) 
 (  X_{},  Y_{})
       \\
     &
       + (n-2) \textstyle  (\xi_{(6)e_i} \eta ) \wedge   (  \xi_{(10) e_i} \eta) 
 (  X_{},  Y_{}) 
  - 2 \textstyle   (   \xi_{{(7)} \; e_i} \eta) \wedge (  \xi_{{(9)} \; e_i}
     \eta)(  X_{},  Y_{})
              \\
              &
   - 2 \textstyle  (   \xi_{{(7)} \; e_i} \eta) \wedge (  \xi_{{(10)} \; e_i}   \eta)(  X_{},  Y_{})
  - 2  (   \xi_{{(8)} \; e_i} \eta) \wedge (  \xi_{{(10)} \; e_i}    \eta)(  X_{},  Y_{})
      \\
 &
+ 2\textstyle  (\xi_{(7)X_{}} \eta)(\xi_{(11)\zeta} Y_{})   
- 2\textstyle  (\xi_{(7)Y_{}} \eta)(\xi_{(11)\zeta}X_{}). 
   \end{align*}
   }
  \end{lemma}
  \begin{proof}
   The proof  follows in a similar way as in Lemma \ref{lambdaunouno}.    
   We firstly consider   $d^2 F$ written   in terms of $\nabla^{\Lie{U}(n)}$ and $\xi$, i.e. \eqref{d2F}. Then contracting  with $F$ on the first two arguments, it is obtained the two-form $\eqref{usouso}$. Finally, we  will  compute the $\lcf \lambda^{2,0}\rcf$-component of such a two-form, which  determines the $\lcf \lambda^{2,0}\rcf$-component  of $d^2F$. As before, the required identity follows by using the properties of the components $\xi_{(i)}$ of $\xi$ and the fact $d^2F=0$.  The identity displayed here, with the additional assumptions $\xi_{(11)}=0$ and  $\xi_{(12)}=0$, was already proved in 
    \cite[Lemma 4.5]{GDMC1} where other  additional details can be found. 
  \end{proof}
 
In  previous Lemma, if we use the identity 
 $$ 
 \textstyle (n-1)   \langle (\nabla^{\Lie{U}(n)}_{e_i} \xi_{(4)} )_{e_i} X_{},Y_{} \rangle   
=
  \textstyle ( d (\xi_{(4)e_i} e_i)^\flat)_{\lcf \lambda^{2,0} \rcf} (X,Y)  - 2  \langle \xi_{(1) \xi_{(4)e_i} e_i} X, Y \rangle 
  +   \textstyle  \langle \xi_{(2) \xi_{(4)e_i} e_i} X, Y \rangle,  
  $$
we will obtain   the $\lcf \lambda^{2,0} \rcf$-component of the exterior derivative of the Lee form $\theta$.
\begin{proposition} \label{divergenciadoscero}
  For almost contact metric manifolds of dimension $2n+1$, $n>1$, the
following identity is satisfied
\begin{equation*}
{\rm  \small
     \begin{array}{rl}
           \tfrac{n-2}{2} \textstyle  d\theta_{\lcf \lambda^{2,0} \rcf} (X,Y) 
     = &
      - 3 \textstyle  \langle (\nabla^{\Lie{U}(n)}_{e_i} \xi_{(1)} )_{e_i} X_{}, Y_{} \rangle
   + \textstyle  \langle (\nabla^{\Lie{U}(n)}_{e_i} \xi_{(3)} )_{e_i} X_{}, Y_{} \rangle
  + \textstyle   \langle \xi_{{(3)}X_{}} e_i,  \xi_{{(1)} Y_{}} e_i \rangle
    \\[2mm]
&
  - \textstyle   \langle  \xi_{{(3)}Y_{}} e_i,  \xi_{{(1)}  X_{} }e_i \rangle
  -  \frac12\textstyle   \langle  \xi_{{(3)}X_{}} e_i,  \xi_{{(2)} Y_{}} e_i \rangle
   + \textstyle  \frac12    \langle  \xi_{{(3)}Y_{}} e_i,  \xi_{{(2)}  X_{}} e_i \rangle
    \\[2mm]
   &
 \textstyle  
  + \tfrac{3(n-3)}{2}  \textstyle  \langle \xi_{{(1)}
 \theta^\sharp }X_{}, Y_{} \rangle
        - \textstyle \frac{n-1}{2}  \langle \xi_{{(3)} 
      \theta^\sharp } X_{}, Y_{} \rangle
  
     - \textstyle d^*F(\zeta)  
      \langle  \xi_ {(11)\zeta}  X_{},  \varphi Y_{} \rangle
        \\[2mm]
     &
          - \tfrac{n-2}{2n} d^* \eta     (  \xi_{(10)  X_{}} \eta)  (   Y_{})
         + \tfrac{n-2}{n} d^* F (\zeta) \textstyle     (  \xi_{(10) \varphi X} \eta)  ( \varphi   Y_{}) 
   \\[2mm]
     &
  + 2 \textstyle   (   \xi_{{(7)} \; e_i} \eta) \wedge (  \xi_{{(9)} \; e_i}
     \eta)(  X_{},  Y_{})
   + 2 \textstyle   (   \xi_{{(7)} \; e_i} \eta) \wedge (  \xi_{{(10)} \; e_i}   \eta)(  X_{},  Y_{})
      \\[2mm]
     &
  + 2  (   \xi_{{(8)} \; e_i} \eta) \wedge (  \xi_{{(10)} \; e_i}    \eta)(  X_{},  Y_{})
- 2\textstyle  (\xi_{(7)X_{}} \eta)(\xi_{(11)\zeta} Y_{})   
+ 2\textstyle  (\xi_{(7)Y_{}} \eta)(\xi_{(11)\zeta}X_{}). 
   \end{array}}
  \end{equation*}
  \end{proposition}

  Next we give a third   consequence of $d^2 F=0$.  
   \begin{lemma} \label{etameanid1}
 For almost contact metric manifolds of dimension $2n+1$,  the following
identity is satisfied
  \begin{eqnarray*}
0 & = &
    - \textstyle \langle (\nabla^{\Lie{U}(n)}_{\zeta} \xi_{(4)})_{e_i} { e_i},  X \rangle
 - (n-1)   \textstyle   ((\nabla^{\Lie{U}(n)}_{e_i} \xi_{(5)})_{e_i} {\eta})  (X)
      +  \textstyle    ((\nabla^{\Lie{U}(n)}_{e_i} \xi_{(8)})_{e_i} {\eta}) (X)
      \\
      &&
        - \textstyle     ((\nabla^{\Lie{U}(n)}_{e_i} \xi_{(10)})_{ e_i} {\eta}) (X)
   + \textstyle    \langle (\nabla^{\Lie{U}(n)}_{e_i} \xi_{(11)})_{\zeta}   { e_i} , X \rangle
               -  \textstyle   (\xi_{(8)e_i} {\eta})(\xi_{(3)e_i} X)
               \\
               &&
              -  \textstyle    (\xi_{(7)e_i} {\eta})(\xi_{(3)e_i} X) 
       + \textstyle    (\xi_{(10)e_i} {\eta})(\xi_{(1)X} e_i)
              -	 \tfrac12  \textstyle   (\xi_{(10)e_i} {\eta})(\xi_{(2)X} e_i)
         \\
         &&
           +  \textstyle    \langle  \xi_{(11)\zeta}   {e_i} , \xi_{(1)X}  e_i \rangle 
 - \textstyle   \tfrac12    \langle  \xi_{(11)\zeta}   {e_i} , \xi_{(2)X}  e_i \rangle 
        +  \textstyle     (\xi_{(5)   \xi_{(4)e_i} { e_i}} \eta ) (X)
        \\
        &&
        - \tfrac{1}{n-1} \textstyle    (\xi_{(8)   \xi_{(4)e_i} { e_i}} \eta ) (X)
         + \textstyle    \langle  (\xi_ {(9)\xi_{(4)e_i} {e_i}}  {\eta}) (X) 
   - \textstyle    \langle  (\xi_ {(6)\xi_{(4)e_i} {e_i}}  {\eta}) (X) 
   \\
   &&
     - \textstyle  
      (\xi_{(7)   \xi_{(4)e_i} { e_i}} \eta ) (X)
- (n-1) \textstyle     (\xi_{(5)\xi_\zeta \zeta} \eta)(X)
- \textstyle   ( \xi_{(10) \xi_{\zeta} \zeta} \eta)(X)
     - \textstyle   \langle  \xi_{(11)\zeta}   X , \xi_\zeta \zeta  \rangle.
       \end{eqnarray*}      
\end{lemma}
 \begin{proof}
   It follows in a similar way  as in the proof of Lemma  \ref{lambdaunouno}.
     The form  $d^2 F$ is written in terms of $\nabla^{\Lie{U}(n)}$ and $\xi$, i.e.  \eqref{d2F}. Then, doing  a contraction by $F$, it is obtained the two-form $\eqref{usouso}$.    Since $\lcf \lambda^{1,0}\rcf$-part of $d^2F$ is determined by $\lcf \lambda^{1,0}\rcf$-component of such a two-form, this component  vanishes because $d^2F=0$ and the required identity is obtained.  Such an identity, with the additional assumptions $\xi_{(11)}=0$ and $\xi_{(12)}=0$, was already showed in  
     \cite[Lemma 4.6, page 165]{GDMC1}. However,  we have noted some mistakes in the computation there.  This is the reason why there are  differences with the first  identity  given there in the mentioned Lemma 4.6.     
   \end{proof}  
Next by  noting  that  $ (\nabla^{\Lie{U}(n)}_X \xi_{(4)})_{e_i} e_i =  \nabla^{\Lie{U}(n)}_X \xi_{(4) e_i} e_i$ and using the identities   
\begin{gather*}
(\xi_{(5)X} \eta) (Y_{})  =  \tfrac{d^*\eta}{2n}( \langle X_{} , Y_{} \rangle - \eta(X) \eta(Y)), 
\quad (\xi_{(6)X} \eta) (Y)  =  -\tfrac{d^*F(\zeta)}{2n} F( X, Y),
\\
( (\nabla^{\Lie{U}(n)}_X  \xi_{(5)})_Y \eta) (Z_{}) = \tfrac{d(d^*\eta)(X)}{2n}( \langle Y_{} , Z_{} \rangle - \eta(Y) \eta(Z)), \\
\textstyle \langle (\nabla^{\Lie{U}(n)}_{\zeta} \xi_{(4)})_{ e_i} { e_i},  X \rangle = d (\xi_{(4)e_i} { e_i})^\flat (\zeta , X) - \langle \xi_{(11)\zeta} X, \xi_{(4)e_i} { e_i} \rangle - (\xi_{\varphi^2 X} \eta)(\xi_{(4)e_i} { e_i}),
 \end{gather*}
  another version  of the  identity in the previous Lemma  is given  in next Proposition. Such a version relates the exterior derivatives of the Lee form   $\theta$ and the coderivative $d^*\eta$.
\begin{proposition}   \label{divergencia}
For almost contact metric manifolds of dimension $2n+1$,  we have 
  \begin{align*}
 		     \textstyle  \tfrac{n-1}{2} d \theta (\zeta , X)
		     		       = &
          \tfrac{n-1}{2n} d(d^* \eta) (\varphi^2 X) 
         +  \textstyle    ((\nabla^{\Lie{U}(n)}_{e_i} \xi_{(8)})_{e_i} {\eta}) (X)
        - \textstyle      ((\nabla^{\Lie{U}(n)}_{e_i} \xi_{(10)})_{ e_i} {\eta}) (X)
 \\
   &
   + \textstyle    \langle (\nabla^{\Lie{U}(n)}_{e_i} \xi_{(11)})_{\zeta}   { e_i} , X \rangle     
                  -  \textstyle   (\xi_{(7)e_i} {\eta})(\xi_{(3)e_i} X)   
               -  \textstyle   (\xi_{(8)e_i} {\eta})(\xi_{(3)e_i} X)
                                  \\
& 
       + \textstyle    (\xi_{(10)e_i} {\eta})(\xi_{(1)X} e_i)
              -	\textstyle  \tfrac12      (\xi_{(10)e_i} {\eta})(\xi_{(2)X} e_i)
           +  \textstyle    \langle  \xi_{(11)\zeta}   {e_i} , \xi_{(1)X}  e_i \rangle 
           \\
           &
 - \textstyle  \tfrac12    \langle  \xi_{(11)\zeta}   {e_i} , \xi_{(2)X}  e_i \rangle 
        - \tfrac{n}{2} \textstyle    (\xi_{(8) 
        \theta^\sharp } \eta ) (X) 
         + \textstyle    \frac{n-1}{2}\langle  (\xi_ {(10)
           \theta^\sharp }  {\eta}) (X) 
         \\
         &
          +  \textstyle  \tfrac{n-1}{2}  
          \theta( \xi_{(11)\zeta}   X) 
- \tfrac{n-1}{2n} d^*\eta (\xi_\zeta \eta)(X) 
- \textstyle   ( \xi_{(10) \xi_{\zeta} \zeta} \eta)(X)
     - \textstyle   \langle  \xi_{(11)\zeta}   X , \xi_\zeta \zeta  \rangle.
       \end{align*}      
   In particular, if the  almost contact metric structure is  of type $\mathcal{C}_1 \oplus \mathcal{C}_2 \oplus \mathcal{C}_3 \oplus    \mathcal{C}_5    \oplus \mathcal{C}_6 \oplus \mathcal{C}_9 \oplus  \mathcal{C}_{12} $ or $\mathcal{C}_1 \oplus \mathcal{C}_2 \oplus \mathcal{C}_5 \oplus    \mathcal{C}_6    \oplus \mathcal{C}_7 \oplus \mathcal{C}_9 \oplus  \mathcal{C}_{12} $  and  $n>1$,  then  $d (d^* \eta)$ is given  by 
   $d (d^* \eta) = - d^*\eta \,\xi_\zeta \eta + d (d^* \eta)(\zeta) \eta $. Likewise,  for the type 
   $\mathcal{C}_1 \oplus \mathcal{C}_2 \oplus \mathcal{C}_3 \oplus    \mathcal{C}_5     \oplus \mathcal{C}_9 \oplus  \mathcal{C}_{12} $ and $n > 1$,   the one-form $\mathrm{div } (\zeta) \; \eta = - d^*\eta \, \eta$ is closed.
\end{proposition}

If  we consider the identity $d^2 \eta=0$, we will  obtain an expression for $d(d^*F(\zeta))$.
\begin{proposition} \label{mainid} 
For almost contact metric  manifolds of dimension $2n+1$, the exterior derivative $d(d^*F(\zeta))$ is given by 
\begin{align*}
  \textstyle  \frac{n-1}{2n} d(d^* F(\zeta))(X_{\zeta^\perp})  = & 
     \textstyle    ((\nabla^{\Lie{U}(n)}_{e_i} \xi_{(7)})_{ e_i} \eta))(\varphi X)
   -   \textstyle     ((\nabla^{\Lie{U}(n)}_{e_i} \xi_{(10)})_{ e_i} \eta)(\varphi X)
 -  \textstyle       (\xi_{(7)e_i} \eta)(\xi_{(3)e_i} \varphi  X) 
\\
&
- 2\textstyle        (\xi_{(10) e_i} \eta)( \xi_{(1)\varphi X} e_i )  
-  \textstyle   \frac12     (\xi_{(10) e_i} \eta)( \xi_{(2)\varphi X} e_i ) 
  -\frac{n-1}{2n}     d^* F(\zeta)  \theta( X)   
  \\
  &
-   \textstyle  \tfrac{n-1}{2}  ( \xi_{(7) \theta^\sharp } \eta)(\varphi  X)
+ \textstyle     \frac{n-2}{2} ( \xi_{(10) \theta\sharp } \eta)(\varphi X ) 
+  \textstyle  \frac{n-1}{2n} d^* F(\zeta) ( \xi_{(12)\zeta  }  \eta ) ( X)
\\
&
-  (\xi_{(7)  \xi_{(12)\zeta  } \zeta} \eta ) ( \varphi X )
+  (\xi_{(10)  \xi_{(12)\zeta  } \zeta} \eta ) ( \varphi X ),
\end{align*}
\begin{align*}
 d(d^*F(\zeta))(\zeta) =&   
   \textstyle  + \frac{1}{n} d^*\eta d^*F(\zeta)  - \langle d \xi_\zeta \eta , F \rangle      
       + \textstyle  2     (\xi_{(7) e_i} \eta)  (\varphi \xi_{(8)e_i} {\zeta}) 
       + 2      (\xi_{(10) e_i} \eta) (\varphi  \xi_{(11)\zeta} e_i).      
\end{align*}
In particular,  if the almost contact metric structure is of type $\mathcal{C}_1 \oplus \mathcal{C}_2 \oplus \mathcal{C}_3 \oplus \mathcal{C}_4 \oplus  \mathcal{C}_5 \oplus \mathcal{C}_6 \oplus \mathcal{C}_8 \oplus \mathcal{C}_9 \oplus \mathcal{C}_{11} \oplus \mathcal{C}_{12} $ and $n>1$, then the exterior derivative  $d (d^*F(\zeta))$ is given  by  
\begin{equation} \label{dstarefezeta}
 d (d^*F(\zeta))  =  -  d^*F(\zeta) \textstyle   \theta
 +  d^*F(\zeta) \xi_\zeta \eta  
 + \left( \tfrac{1}{n} d^* \eta \;  d^*F(\zeta) - \langle d \xi_\zeta \eta , F \rangle \right) \eta.
\end{equation}

\end{proposition}
\begin{proof}
Using $\nabla = \nabla^{\Lie{U}(n)} - \xi$ and   $\nabla^{\Lie{U}(n)} \eta =0$,  the form    $d^2 \eta$ is written in terms of $\nabla^{\Lie{U}(n)}$ and $\xi$, i.e.
  \begin{eqnarray*}
  d^2\eta (X_1,X_2,X_3) & = & \textstyle \sum_{1 \leq a < b \leq 3} (-1)^{a+b} \left( \left((\nabla^{\Lie{U}(n)}_{X_a} \xi)_{X_b} - (\nabla^{\Lie{U}(n)}_{X_b} \xi)_{X_a}\right) \eta \right) (X_c) \nonumber
   \\
&& 
\textstyle + \sum_{1 \leq a < b \leq 3} (-1)^{a+b}(\xi_{\xi_{X_a} X_b - \xi_{X_b} X_a}  \eta) (X_c) \\
&&
\textstyle - \sum_{1 \leq a < b \leq 3} (-1)^{a+b}([\xi_{X_a} ,\xi_{X_b}]  \eta) (X_c, X_d), \nonumber
  \end{eqnarray*}
 where  $ \{c\} =\{1,2,3\}-\{a,b\}$ in each case.
 
 Note that,  $\Lambda^3 T^* M$  is decomposed into 
\begin{equation} \label{tresformas}
\Lambda^3 T^* M = \lcf \lambda^{3,0} \rcf + \lcf \lambda^{2,1}_0 \rcf + \lcf \lambda^{1,0} \rcf  \wedge F + \lcf \lambda^{2,0} \rcf  \wedge \eta + \mathbb{R}   F \wedge \eta  + [\lambda^{1,1}_0 ]\wedge \eta,
\end{equation}
under the action of $\Lie{U}(n) \times   1$. Now, we compute  the parts of $d^2 \eta$ in   $ \lcf \lambda^{1,0} \rcf  \wedge F$ and $\mathbb{R} \; F \wedge  \eta$, by  
 contracting   $d^2 \eta$ with $F$ on the first to arguments. Then  we have 
\begin{align}
0  = & \textstyle \tfrac12   d^2 \eta (e_i, \varphi e_i, X )
    =   -  \textstyle   ((\nabla^{\Lie{U}(n)}_{e_i} \xi)_{\varphi e_i} \eta)(X)
              +  \textstyle    ((\nabla^{\Lie{U}(n)}_{e_i} \xi)_{X} \eta)( \varphi e_i)
              -  \textstyle   ((\nabla^{\Lie{U}(n)}_{X} \xi)_{ e_i} \eta)( \varphi e_i) \nonumber\\
     & \hspace{2.5cm} \;\;\;\;\; +   \textstyle    (\xi_{e_i} \eta)(\xi_{\varphi e_i} X)   
             +   \textstyle   (\xi_{ \xi_{e_i} X } \eta)( \varphi e_i)        
             -    \textstyle   (\xi_{e_i } \eta)( \xi_{X}  \varphi e_i) 
                 \label{aboveiden} 
                    \\
    &  
\hspace{2.5cm}\;\;\; \;\; -   \textstyle    ( \xi_{X} e_i   \eta)( \varphi e_i)    -   \textstyle   (\xi_{\xi_{e_i} \varphi e_i} \eta)(X)
         +   \textstyle    (\xi_{X }\eta)(\xi_{e_i} \varphi e_i).  \nonumber                            
\end{align}
Now, by considering $X= X_{\zeta^\perp} $ and  each component $\xi_{(i)}$ of the intrinsic torsion beside with its properties, it follows the first required identity. 
Here it is used 
$$
  (\nabla^{\Lie{U}(n)}_{Z} \xi_{(6)})_{X} Y  =  \tfrac{1}{2n} d(d^*F(\zeta))(Z) (F(X,Y) \zeta + \eta(Y) \varphi(X)).
$$
From this, it is obtained
$
\textstyle  ( (\nabla^{\Lie{U}(n)}_{e_i} \xi_{(6)})_{e_i} \eta)(Z) = - \frac{1}{2n} d(d^*F(\zeta))(\varphi Z).
$

For the second identity required  in Lemma, by  taking  $X=\zeta$ in the identity \eqref{aboveiden}
 and considering  each component $\xi_{(i)}$ of the intrinsic torsion beside with its properties, it follows the second required identity. 
In the computation it is used \vspace{1mm}

$\textstyle    ((\nabla^{\Lie{U}(n)}_{e_i} \xi_{(12)})_{\zeta} \eta) (\varphi e_i)  
 =  
\textstyle 
- \langle d\xi_\zeta \eta, F \rangle -   \langle \xi_{(4)e_i} e_i, \varphi \xi_\zeta \zeta \rangle 
 =  - \mathrm{div} \, \varphi \xi_\zeta \eta) + \langle \xi_{(4)e_i} e_i, \varphi \xi_\zeta \zeta \rangle.
$
\end{proof}

Next, using the identity $d^2 \eta =0$, we will study the contributions of the components $\xi_{(i)}$   in the exterior derivative $d \xi_\zeta \eta$. Firstly we give an expression for $d \xi_\zeta \eta$ in terms of $\nabla^{\Lie{U}(n)}$ and $\xi$.
\begin{proposition} \label{extderxi12}
For almost contact metric manifolds, the exterior derivative of the one-form $\xi_\zeta \eta$
 which determines the $\mathcal{C}_{12}$-component of the intrinsic torsion is given by
  \begin{align*}
 d \xi_{\zeta}  \eta (X, Y) =& 
      ((\nabla^{\Lie{U}(n)}_{\zeta} \xi)_X  \eta) (Y )
      -((\nabla^{\Lie{U}(n)}_{\zeta} \xi)_Y  \eta) ( X)
    +  (\xi_X \eta) (\xi_{\zeta} Y ) 
    - (\xi_Y \eta) (\xi_{\zeta} X )
         \\
     &  
      - (\xi_{\xi_X {\zeta}} \eta)(Y) 
       + (\xi_{\xi_Y {\zeta}} \eta)(X)  
            + (\xi_{\xi_{\zeta} X} \eta)(Y)  
        - (\xi_{\xi_{\zeta} Y} \eta)(X) 
\end{align*}
\end{proposition} 
\begin{proof}
It follows from the identity $d^2 \eta (X,Y,\zeta)=0$. In fact, we have 
\begin{align*}
(d^2 \eta) (X,Y,{\zeta})
=& 
 ((\nabla^{\Lie{U}(n)}_X \xi)_{\zeta}  \eta) (Y) 
 -((\nabla^{\Lie{U}(n)}_Y \xi)_{\zeta} \eta) ( X ) 
     -((\nabla^{\Lie{U}(n)}_{\zeta} \xi)_X  \eta) (Y )
      +((\nabla^{\Lie{U}(n)}_{\zeta} \xi)_Y  \eta) ( X)
  \\ 
& 
 + (\xi_{\zeta} \eta) (\xi_X Y )
   -  (\xi_{\zeta} \eta) (\xi_Y X ) 
    -  (\xi_X \eta) (\xi_{\zeta} Y ) 
    + (\xi_Y \eta) (\xi_{\zeta} X )
\\
&
      + (\xi_{\xi_X {\zeta}} \eta)(Y) 
       - (\xi_{\xi_Y {\zeta}} \eta)(X)  
        - (\xi_{\xi_{\zeta} X} \eta)(Y)  
        + (\xi_{\xi_{\zeta} Y} \eta)(X). 
\end{align*}
Since we have $((\nabla^{\Lie{U}(n)}_X \xi)_{\zeta}  \eta) (Y) = (\nabla^{\Lie{U}(n)}_X \xi_{\zeta}  \eta) (Y)$ and 
$$
d \xi_\zeta \eta(X,Y)  =  (\nabla^{\Lie{U}(n)}_X \xi_{\zeta}  \eta) (Y) - (\nabla^{\Lie{U}(n)}_Y \xi_{\zeta}  \eta) (X) 
 + (\xi_{\zeta} \eta) (\xi_X Y )
 -  (\xi_{\zeta} \eta) (\xi_Y X ),
 $$
it is obtained the required identity.
\end{proof}

In next lemma, $(d \xi_\zeta \eta)_V$ will denote the projection of $d \xi_\zeta \eta$ on the $\Lie{U}(n)$-space $V$.
\begin{lemma} \label{comp12}
The $\Lie{U}(n)$-components of $d \xi_\zeta \eta$ are given by: 

 \noindent $(d \xi_\zeta \eta)_{\mathbb{R} \, F} = \frac1{n} \langle d \xi_\zeta \eta, F \rangle \; F$, where 
\begin{align*}
\qquad  \langle d \xi_\zeta \eta, F \rangle 
 =& 
       -  d(d^*F(\zeta))(\zeta)  + \textstyle \frac{1}{n} d^*\eta \, d^*F(\zeta)
      + 2 \textstyle   ( \xi_{(7)  e_i} \eta)  (\varphi \xi_{(8)e_i} \zeta)   
        + 2 \textstyle   (\xi_{(10) e_i} \eta) (\varphi \xi_{(11)\zeta} e_i ),  
\end{align*}
\begin{align*}
 (d \xi_{\zeta}  \eta)_{[\lambda^{1,1}] }  (X, Y) =& 
      -  \textstyle  \tfrac{1}{n} d(d^*F(\zeta))(\zeta) F(X, Y)
     +  \textstyle  \tfrac{1}{n^2} d^*\eta d^*F(\zeta) F(X, Y)
     +2((\nabla^{\Lie{U}(n)}_{\zeta} \xi_{(7)})_{X}  \eta) (Y )
    \\
    & - \tfrac{2}{n} d^* \eta (\xi_{(7) X} \eta) (Y)
      +  \textstyle  \tfrac{2}{n} d^*F(\zeta)  (\xi_{(8) X} \eta) (\varphi Y)
      - 2   (\xi_{(7) X} {\eta})  (\xi_{(8) Y} \zeta )  
      \\
      &
       +  2    ( \xi_{(7) Y} {\eta})  (\xi_{(8) X} \zeta) 
        +  2  (\xi_{(9) X} \eta)  (\xi_{(10) Y} {\zeta}) 
     - 2  (\xi_{(9) Y} \eta)  (\xi_{(10) X} {\zeta})
      \\
     &
    + 2 (\xi_{(10)X} \eta) (\xi_{(11)\zeta} Y ) 
    - 2 (\xi_{(10)Y} \eta) (\xi_{(11)\zeta} X ),     
       \end{align*}       
      \begin{align*}
  \hspace{-11mm} 
  (d\xi_{\zeta}  \eta)_{\lcf \lambda^{2,0} \rcf } ( X, Y)  = & 
    2((\nabla^{\Lie{U}(n)}_{\zeta} \xi_{(10)})_{X}  \eta) (Y )
   -   \textstyle  \tfrac2{n} d^*F(\zeta) \langle \xi_{(11)\zeta} X , \varphi Y \rangle
      - 2   (\xi_{ (7) X} \eta)   (\xi_{ (9) Y} \zeta)
      \\
      &
       + 2   (\xi_{ (7) Y} \eta)  ( \xi_{ (9) X} \zeta)
    + 2  (\xi_{(7)X} \eta) (\xi_{(11)\zeta} Y ) 
    - 2  (\xi_{(7)Y} \eta) (\xi_{(11)\zeta} X  )
\\
& - \tfrac{2}{n} d^*\eta (\xi_{(10) X} \eta)(Y)
        + 2  (\xi_{ (8) X} \eta)   (\xi_{ (10) Y} \zeta)  
       - 2   ( \xi_{ (8) Y} \eta)   (\xi_{ (10) X} \zeta),   
   \end{align*}

 \noindent $(d \xi_{\zeta}  \eta)_{\eta \wedge \lcf \lambda^{1,0} \rcf} = \eta \wedge \zeta \lrcorner d \xi_\zeta \eta$, where  
\begin{align*}
  \zeta \lrcorner d \xi _\zeta \eta (X)  
  =&    ((\nabla^{\Lie{U}(n)}_{\zeta} \xi_{(12)})_\zeta  \eta) (X )  +  (\xi_{(12)\zeta} \eta)  ( \xi_{(11)\zeta} X)  
    -  \textstyle   \tfrac{1}{2n} d^*\eta   (\xi_{(12)\zeta} \eta )  (X) 
    \\
    & 
     - (\xi_{(8) \xi_{(12)\zeta} \zeta} \eta) (X) 
-  (\xi_{(9) \xi_{(12)\zeta} \zeta} \eta) (X)  
 -  \textstyle  \tfrac{1}{2n} d^*F(\zeta) (\xi_{(12)\zeta} \eta )  (\varphi X)  
 \\
 &
 + (\xi_{(7) \xi_{(12)\zeta} \zeta} \eta) (X) 
+  (\xi_{(10) \xi_{(12)\zeta} \zeta} \eta) (X). 
 \end{align*}
\end{lemma}
\begin{proof}
It follows  from the expression for $d\xi_\zeta \eta$ given in Proposition \ref{extderxi12}.
\end{proof}

Now we give   some  sufficient conditions  for the vanishing of   components of $d\xi_\zeta \eta$.
\begin{proposition}$\,$
   \begin{enumerate}
  \item[$(\mathrm{i})$]
If the structure is of type $\mathcal{C}_1 \oplus \mathcal{C}_2 \oplus \mathcal{C}_3 \oplus    \mathcal{C}_4 \oplus \mathcal{C}_5  \oplus \mathcal{C}_9 \oplus \mathcal{C}_{12} \oplus \mathcal{C}_x \oplus \mathcal{C}_{y}  $,
 where $(x,y) \in \{ (7,10), (7,11), (8,10), (8,11)\}$,   then  $\langle d\xi_\zeta \eta, F \rangle  =0$.  In such a case, we have $\mathrm{div} \varphi \xi_\zeta \zeta = (n-1) \theta  (\varphi \xi_\zeta \zeta) $.
 \item[$(\mathrm{ii})$]
If the structure is of type $\mathcal{C}_1 \oplus \mathcal{C}_2 \oplus \mathcal{C}_3 \oplus    \mathcal{C}_4 
\oplus \mathcal{C}_5
\oplus \mathcal{C}_x   \oplus \mathcal{C}_9   \oplus \mathcal{C}_{11} \oplus \mathcal{C}_{12}$ or 
$\mathcal{C}_1 \oplus \mathcal{C}_2 \oplus \mathcal{C}_3 \oplus    \mathcal{C}_4 
\oplus \mathcal{C}_5 
\oplus \mathcal{C}_x      \oplus \mathcal{C}_{10} \oplus \mathcal{C}_{12}$, where $x\in \{6,8\}$,   then   $(d\xi_\zeta \eta)_{[ \lambda^{1,1}_0]}  =0$.

\item[$(\mathrm{iii})$]
If the structure is of type $\mathcal{C}_1 \oplus \mathcal{C}_2 \oplus \mathcal{C}_3 \oplus    \mathcal{C}_4 
\oplus \mathcal{C}_5
\oplus \mathcal{C}_8   \oplus \mathcal{C}_9   \oplus \mathcal{C}_{11} \oplus \mathcal{C}_{12}$ or 
$\mathcal{C}_1 \oplus \mathcal{C}_2 \oplus \mathcal{C}_3 \oplus    \mathcal{C}_4 
\oplus \mathcal{C}_5 
\oplus \mathcal{C}_8      \oplus \mathcal{C}_{10} \oplus \mathcal{C}_{12}$,    then   $(d\xi_\zeta \eta)_{[ \lambda^{1,1}]}  =0$.

\item[$(\mathrm{iv})$]
If the structure is of type $\mathcal{C}_1 \oplus \mathcal{C}_2 \oplus \mathcal{C}_3 \oplus    \mathcal{C}_4 
\oplus \mathcal{C}_5
    \oplus \mathcal{C}_{8} \oplus \mathcal{C}_{12}  \oplus \mathcal{C}_x   \oplus \mathcal{C}_y$, 
$(x,y) \in \{ (6,7), (6,9) , (9,11)\}$, then   $(d\xi_\zeta \eta)_{\lcf \lambda^{2,0} \rcf}  =0$.

\item[$(\mathrm{v})$]
If the structure is of type $\mathcal{C}_1 \oplus \mathcal{C}_2 \oplus \mathcal{C}_3 \oplus    \mathcal{C}_4 \oplus \mathcal{C}_5  \oplus 
\mathcal{C}_8 \oplus \mathcal{C}_9 \oplus \mathcal{C}_{11} \oplus \mathcal{C}_{12}$,   then   $d(\xi_\zeta \eta)_{[ \lambda^{1,1}]}  =0$ and   $(d\xi_\zeta \eta)_{\lcf \lambda^{2,0} \rcf}  =0$. That is, $d\xi_\zeta \eta \in \eta \wedge \lcf \lambda^{1,0} \rcf$.

\end{enumerate}
\end{proposition}
\begin{proof}
All parts are direct consequences of Lemma \ref{comp12}.
\end{proof}

\section{Non-existence of certain types of almost contact metric structures}\setcounter{equation}{0}
The main purpose of this section is to prove that there are certain types of  almost contact metric structure,  which were initially possible from algebraic point of view,  that  do not exist because of geometry. A first result in this sense was proved by Marrero in  \cite{JCMAR}. He showed that if a connected  manifold of dimension higher than $3$ is equipped with an almost contact metric structure of  type $\mathcal{C}_5 \oplus \mathcal{C}_6$, then it must be of one of  the singles types $\mathcal{C}_5$ or  $\mathcal{C}_6$. This is a particular case of the result proved here.

\begin{theorem} \label{mainmain}  For a connected almost contact metric manifold of dimension $2n+1$, $n>1$,  we have:
\vspace{2mm}

\noindent $(i)$  If the  structure is of type $\mathcal{C}_1 \oplus \mathcal{C}_2 \oplus \mathcal{C}_3 \oplus  \mathcal{C}_5 \oplus \mathcal{C}_6 \oplus \mathcal{C}_8 \oplus \mathcal{C}_9\oplus \mathcal{C}_{11}\oplus \mathcal{C}_{12}$ with $\langle d\xi_\zeta \eta, F \rangle =0$ and $(\xi_{(5)}, \xi_{(6)} ) \neq (0,0)$, then it is of type  $\mathcal{C}_1 \oplus \mathcal{C}_2 \oplus \mathcal{C}_3  \oplus \mathcal{C}_5 \oplus \mathcal{C}_8 \oplus \mathcal{C}_9 \oplus \mathcal{C}_{11}\oplus \mathcal{C}_{12}$, i.e. $\xi_{(6)}=0$, or of type 
$  \mathcal{C}_2 \oplus   \mathcal{C}_6   \oplus \mathcal{C}_9 \oplus \mathcal{C}_{12} $ with $\xi_{(6)} \neq 0$. Likewise, for this last type $($with the previously fixed condition $\langle d\xi_\zeta \eta, F \rangle =0$$)$,  $d(d^*F(\zeta)) = d^*F(\zeta) \,  \xi_\zeta \eta$  and  $\xi_\zeta \eta$ is closed.
\vspace{2mm}

\noindent $(ii)$ If the  structure is of type $\mathcal{C}_2 \oplus \mathcal{C}_5 \oplus \mathcal{C}_7 \oplus  \mathcal{C}_9 $ being $(\xi_{(5)}, \xi_{(7)} ) \neq (0,0)$,  then it is of type  $\mathcal{C}_2 \oplus \mathcal{C}_7 \oplus \mathcal{C}_9$, or of type 
$  \mathcal{C}_2 \oplus   \mathcal{C}_5   \oplus \mathcal{C}_9  $. 
  
\end{theorem}

\begin{proof} 
 For $(i)$, from  \eqref{dstarefezeta}, one has
$
  d(d^* F(\zeta)) =  d^*F(\zeta) \xi_\zeta \eta + \tfrac{1}{n} d^*\eta d^*F(\zeta) \eta. 
$
 This implies 
 \begin{align*} 
 0  = &   \tfrac{1}{n} d^*\eta d^*F(\zeta) \eta \wedge  \xi_\zeta \eta + d^*F(\zeta) d \xi_\zeta \eta   + \tfrac1{n} d^*F(\zeta)  d(d^*\eta) \wedge \eta   + \tfrac{1}{n} d^*\eta d^*F(\zeta)  \xi_\zeta \eta \wedge \eta  
   \\
   &
    +\tfrac{1}{n^2} d^* \eta   d^* F( \zeta)^2  F + \tfrac{1}{n} d^*\eta d^*F(\zeta)  \xi_\zeta \eta \wedge \eta.
\end{align*}
Therefore,
\begin{align*} 
 0  =   
  d^*F(\zeta) d \xi_\zeta \eta 
   + \tfrac1{n} d^*F(\zeta)  d(d^*\eta) \wedge \eta   
    +\tfrac{1}{n^2} d^* \eta   d^* F( \zeta)^2  F + \tfrac{1}{n} d^*\eta d^*F(\zeta)  \xi_\zeta \eta \wedge \eta.
\end{align*}
Since the third  summand must be zero, one has $d^*\eta d^*F(\zeta) =0$. From this it follows 
$$
 d(d^* F(\zeta)) =  d^*F(\zeta) \xi_\zeta \eta, \qquad   d^*F(\zeta) d \xi_\zeta \eta  = 0, \qquad d^*F(\zeta)  d(d^*\eta)  = 0.
$$

On the other hand, since $d^*\eta$ and $d^*F(\zeta)$ determine  $\xi_{(5)}$ and $\xi_{(6)}$, respectively, the condition $(\xi_{(5)}, \xi_{(6)}) \neq (0,0)$ implies $(d^*\eta , d^* F(\zeta) ) \neq (0,0)$. From this and  $d^*\eta d^*F(\zeta) =0$, it follows that set $A$ of those points such that $d^*F(\zeta) =0$ coincides with set of points such $d^*\eta \neq 0$. Hence $A$ is open. Likewise the set $B$ consisting of points such $d^*\eta=0$ coincides with the set of points such that  $d^*F(\zeta) \neq 0$. Thus $B$ is open. Since the manifold is connected, $A$ or $B$ must be empty. Therefore,  $d^*F(\zeta) =0$  on the whole manifold, or $d^*\eta =0$  on the whole manifold.

If  $d^* F(\zeta)$ is non-zero, then   $ d(\xi_\zeta \eta) = 0$, $d^*\eta =0 $ and $ d(d^*\eta)=0$   on the whole connected manifold. Hence $\xi_{(5)}=0$, $\xi_{(6)}\neq 0$ and, using Lemma \ref{comp12},  we have 
   \begin{gather*}
0 =   (d \xi_{\zeta}  \eta)_{[\lambda^{1,1}] }  (X, Y) = 
            \textstyle  \frac{2}{n} d^*F(\zeta)  (\xi_{(8) X} \eta) (\varphi Y),
            \\
            0 = (d\xi_{\zeta}  \eta)_{\lcf \lambda^{2,0} \rcf } ( X, Y)  = 
   -   \textstyle  \frac2{n} d^*F(\zeta) \langle \xi_{(11)\zeta} X , \varphi Y \rangle.
   \end{gather*}      
Therefore, $\xi_{(8)} = 0 $ and   $\xi_{(11)}=0$.  
 
Finally, since $d \eta = \frac{d^*F(\zeta)}{n} F +  \xi_{\zeta} \eta \wedge \eta  $, we have
$
0  = \tfrac{1}{n} d^*F(\zeta) d F.  
$   
 This implies     $\xi_{(1)} =0$ and $\xi_{(3)} =0$ (see comments  about $dF$ in Section \ref{ejemplo}).

 For $(ii)$, from Lemma \ref{lambdaunouno}, we have $d^*\eta (\xi_{(7)} \eta)=0$ which, in this case, is equivalent to $d^*\eta \, d\eta =0$. Then, doing the exterior derivative in both sides and taking Proposition \ref{divergencia} into account, we obtain 
 $
 d(d^*\eta)(\zeta) \eta \wedge d \eta =0, 
 $
 where $d( d^*\eta) = d(d^*\eta)(\zeta) \eta$. If there is a point where $\xi_{(7)} \neq 0$, then $d^* \eta=0$ and  $d( d^*\eta)=0$ on the whole manifold. Thus  $\xi_{(7)} \neq 0$ everywhere with $\xi_{(5)} = 0$ or   $\xi_{(7)} = 0$ everywhere with $\xi_{(5)} \neq 0$.
\end{proof}

The following result is an immediate consequence of previous Theorem.
\begin{corollary}
For a connected almost contact metric manifold of dimension $2n+1$, $n>1$,  if the  structure is of type $\mathcal{C}_1 \oplus \mathcal{C}_2 \oplus \mathcal{C}_3 \oplus  \mathcal{C}_5 \oplus \mathcal{C}_6 \oplus \mathcal{C}_8 \oplus \mathcal{C}_9\oplus \mathcal{C}_{11}$ with  $(\xi_{(5)}, \xi_{(6)} ) \neq (0,0)$, then it is of type  $\mathcal{C}_1 \oplus \mathcal{C}_2 \oplus \mathcal{C}_3  \oplus \mathcal{C}_5 \oplus \mathcal{C}_8 \oplus \mathcal{C}_9 \oplus \mathcal{C}_{11}$ or of type 
$  \mathcal{C}_2 \oplus   \mathcal{C}_6   \oplus \mathcal{C}_9 $.  
\end{corollary}

\begin{remark} \label{nonexist}
{\rm On a connected almost contact metric manifold of dimension $2n+1$, $n>1$,  the non-existence of  structures  of type $\mathcal{C}_1 \oplus \mathcal{C}_2 \oplus \mathcal{C}_3 \oplus  \mathcal{C}_5 \oplus \mathcal{C}_6 \oplus \mathcal{C}_8 \oplus \mathcal{C}_9\oplus \mathcal{C}_{11}$ with  $\xi_{(5)}\neq 0$ and $\xi_{(6)}\neq 0$ implies that $2^6$ types do not exist. On the other hand, by the previous Corollary, the type  $\mathcal{C}_1 \oplus \mathcal{C}_2 \oplus \mathcal{C}_3  \oplus \mathcal{C}_6 \oplus \mathcal{C}_8 \oplus \mathcal{C}_9\oplus \mathcal{C}_{11}$ with $\xi_{(6)} \neq 0$ must be necessarily the type $ \mathcal{C}_2  \oplus \mathcal{C}_6  \oplus \mathcal{C}_9$. From this it is deduced the non-existence of $2^4$ types in each one of the cases:

(i) $\xi_{(6)}\neq 0$, $\xi_{(2)}\neq 0$, $\xi_{(9)}\neq 0$;
\hspace{1cm} (ii) $\xi_{(6)}\neq 0$, $\xi_{(2)}\neq 0$, $\xi_{(9)}= 0$;

(iii) $\xi_{(6)}\neq 0$, $\xi_{(2)} = 0$, $\xi_{(9)} \neq 0$;
\hspace{1cm} 
(iv) $\xi_{(6)}\neq 0$, $\xi_{(2)} = 0$, $\xi_{(9)} = 0$;

\noindent  which are different to the previous ones. Finally, from Theorem \ref{mainmain} (ii), it turns out the non-existence of $2^2$ another different types.
All of this implies that, for higher dimensions, the possible number of classes to really consider is $2^{12} - (2^6+ 2^4 . 2^2 + 2^2) = 3964$. That is, $132$ classes do not properly exist because of  geometry. At the beginning, the number of  algebraically possible classes is $2^{12}=4096$.

 The type $\mathcal{C}_2 \oplus \mathcal{C}_6 \oplus \mathcal{C}_9 $ could be referred to as  \textit{almost $a$-Sasakian} structure, where $a= \frac{d^*F(\zeta)}{n}$ which is constant for this type in case of $n>1$. Almost Sasakian structure, $a=1$, have been  considered in references (see \cite{Ogiue,JanVan}).  
}
\end{remark}

\section{Examples} \label{ejemplo}
In this section we will display some examples. But previously,  we will describe some preliminary material which will help us to understand them. 
We will begin by computing some  components of the intrinsic torsion $\xi$  by using exterior algebra. Thus we will compute  the $\Lie{U}(n)$-components of the exterior derivatives $d \eta$ and $dF$. To complete the  information about $\xi$, we will also  need  to compute   $\Lie{U}(n)$-components  of the Nijenhuis tensor $N_\varphi$ of the tensor $\varphi$.

The form $d \eta$ is in $\Lambda^2 T^*M = \mathbb{R} F \oplus [ \lambda^{1,1}_0] \oplus \lcf \lambda^{2,0}\rcf \oplus  \eta \wedge \lcf \lambda^{1,0}\rcf$. Using the $\Lie{U}(n)$-map $\xi \to \mbox{alt} (\xi \eta) $, where  $\mbox{alt} (\xi \eta)(X,Y) = (\xi_X \eta) (Y) - (\xi_Y \eta)(X) = - (\nabla_X \eta)(Y) + (\nabla_Y \eta)(X) = - d \eta (X,Y)$, the information about some components of $\xi$ is translated to  the components of $d \eta$: $d\eta_{\mathbb{R} F } = - \mbox{alt} (\xi_{(6)} \eta) = - 2 \xi_{(6)} \eta$,  $\;d\eta_{ [ \lambda^{1,1}_0] } = -\mbox{alt} (\xi_{(7)} \eta) = - 2 \xi_{(7)} \eta $ of $\xi_{(7)}$,  $\; d\eta_{ \lcf \lambda^{2,0}\rcf } =- \mbox{alt} (\xi_{(10)} \eta) = - 2 \xi_{(10)} \eta$ and  $\;d\eta_{ \eta  \wedge \lcf  \lambda^{1,0} \rcf  } = - \mbox{alt} (\xi_{(12)} \eta) = \xi_{( 12) \zeta} \eta \wedge \eta$. 

The form  $dF$ is in $\Lambda^3 T^* M$ which is decomposed as in \eqref{tresformas}. In a similar way as above,  by  using $\Lie{U}(n)$-map $\xi \to \mbox{alt} (\xi F) $, where $\mbox{alt} (\xi F)(X,Y,Z) = (\xi_X F) (Y,Z) +  (\xi_Z F) (X,Y) +  (\xi_Y F) (Z,X) = - d F (X,Y,Z)$,  the information about some components of $\xi$ is translated to  the components of $d F$:  
$\; d F_{\lcf \lambda^{3,0}\rcf } = - \mbox{alt} (\xi_{(1)} F) $, $\;d F_{\lcf \lambda^{2,1}_0\rcf } =  - \mbox{alt} (\xi_{(3)} F) $,  $\; d F_{\lcf \lambda^{1,0}\rcf  \wedge   F } = - \mbox{alt} (\xi_{(4)} F)$,  $d F_{\mathbb{R} F\wedge \eta  } = - \mbox{alt} (\xi_{(5)} F)  $, $\; dF_{ [ \lambda^{1,1}_0] \wedge \eta } = - \mbox{alt} (\xi_{(8)} F) $  and $\; dF_{ \lcf \lambda^{2,0} \rcf \wedge \eta} = - \mbox{alt} ( (\xi_{(10)} + \xi_{(11)})  F) $. Note that the  component   $dF_{ \lcf \lambda^{2,0} \rcf \wedge \eta }$  contains 
 partial information of  $\xi_{(10)}$ and partial information of $\xi_{(11)}$. This is because some diagonal of the space $\mathcal{C}_{10} \oplus \mathcal{C}_{11}$ is included in $\mathrm{ker} (\mbox{alt} (\xi F))$ and  $\mbox{alt} (\xi F) (\mathcal{C}_{10} \oplus \mathcal{C}_{11})$ is the  isomorphic image of the  orthogonal complementary in $\mathcal{C}_{10} \oplus \mathcal{C}_{11}$ of the mentioned   diagonal, $dF_{ \lcf \lambda^{2,0} \rcf \wedge \eta } =  \eta \wedge \left( 2 \xi_{(10)} \eta \circ \varphi  - \xi_{(11)} F \right)$.  However,  if one considers together $d\eta_{ \lcf \lambda^{2,0}\rcf }$ and   $dF_{ \lcf \lambda^{2,0} \rcf \wedge \eta}$,   the whole information about  both,  $\xi_{(10)}$ and  $\xi_{(11)}$, is available.
 
 From all of this, $\xi_{(2)}$ and $\xi_{(9)}$ are  the only  components of the intrinsic torsion about we have no information yet. Thus we need also to consider the \textit{Nijenhuis tensor} $N_\varphi$ of $\varphi$. It is  defined by  $N_\varphi(X,Y) = - \varphi^2 [X,Y] - [\varphi X, \varphi Y] + \varphi [\varphi X, Y] + \varphi [X,\varphi Y ]$ and  contains information about $\xi_{(2)}$ and $\xi_{(9)}$. Hence by analyzing $d\eta$, $dF$ and $N_\varphi$, we will completely determine $\xi$ and locate the type of almost contact metric structure.
 
 Next  we  describe some properties of the tensor  $N_\varphi$ (see \cite{ChineaJCMarr}). The tensor   $N_\varphi$ is in $\Lambda^2 \mathrm{T}^* M \otimes \mathrm{T} M$ and satisfies the properties:
 \begin{align*}
 N_\varphi(\varphi X ,\varphi Y)  =&  -N_\varphi (X,Y) + \eta(X) N_\varphi (\zeta , Y) -  \eta(Y) N_\varphi (\zeta,X) 
\\
  &+  d \eta (\varphi X, \varphi Y) \zeta +  d \eta (\varphi^2 X, \varphi^2 Y) \zeta, 
 \\
 \varphi  N_\varphi( X ,  Y)  = & -N_\varphi (X, \varphi Y) +\eta(Y) N_\varphi (\zeta,\varphi X)  + d\eta (\varphi X , \varphi^2 Y) \zeta,
 \\
\eta(N_\varphi (X,Y)) \zeta   = & d \eta (\varphi X , \varphi Y) \zeta. 
 \end{align*}
Note that, in particular, $N_\varphi(\zeta  ,\varphi X )  =  -  \varphi N_\varphi(\zeta  , X )$ and $\eta ( N_\varphi(\zeta  ,  X ))    =  0$. 
If we consider the almost contact structure without metric, just as a $\Lie{GL}(n, \mathbb{C})$-structure, the tensor $N_\varphi$ is in 
$W_1 \oplus  W_2 \oplus  W_3 \oplus W_4$ where
\begin{eqnarray*}
W_1 & = & \{ t \in \Lambda^2 \mathrm{T}^* M \otimes \mathrm{T} M \, | \, t(\varphi X , \varphi Y ) =  - t( X , Y ), \; \;\;  \varphi t( X ,  Y ) =  - t( X , \varphi Y )\},\\
W_2 & = & \{ t \in \Lambda^2 \mathrm{T}^* M \otimes \mathrm{T} M \, | \, t( X ,  Y ) =  \eta(X)  t( \zeta  , Y ) - \eta(Y) t(\zeta , X), \; \;\;  \eta  (t( X ,  Y )) =  0\},\\
W_3 & = & \{ t \in \Lambda^2 \mathrm{T}^* M \otimes \mathrm{T} M \, | \, t( X ,  Y ) =  \eta ( t( X  , Y )) \zeta, \; \;\; t(\varphi X , \varphi Y ) =   t( X , Y ) \}, 
\\
W_4 & = & \{ t \in \Lambda^2 \mathrm{T}^* M \otimes \mathrm{T} M \, | \, t( X ,  Y ) =  \eta ( t( X  , Y )) \zeta, \; \;\; t(\varphi X , \varphi Y ) =  - t( X , Y ) \}. 
\end{eqnarray*}
The  $\Lie{GL}(n, \mathbb{C})$-components of $N_\varphi$ are given by 
 \begin{eqnarray*}
 N_{\varphi W_1}(X, Y ) & = & N_\varphi (X,Y) -  \eta(X) N_\varphi (\zeta , Y) +  \eta(Y) N_\varphi (\zeta , X) -  d \eta (\varphi X , \varphi Y) \zeta,
 \\
N_{\varphi W_2}(X, Y ) & = & \eta(X) N_\varphi (\zeta , Y) -  \eta(Y) N_\varphi (\zeta , X) ,
\\
N_{\varphi W_3 }(X, Y ) & = &   \tfrac12 (d \eta (\varphi X , \varphi Y) + d \eta (\varphi^2 X , \varphi^2 Y)) \zeta, \\
N_{\varphi W_4 }(X, Y ) & = &  \tfrac12 (d \eta (\varphi X , \varphi Y) - d \eta (\varphi^2 X , \varphi^2 Y)) \zeta. 
\end{eqnarray*}
\begin{remark}
{\rm For an almost contact structure, the structure tensor (a notion defined in \cite{Fuji1,Fuji2}) is determined by the part $T_{\lie{c}}$ of the torsion $T$ of a $\Lie{GL}(n, \mathbb{C})$-connection in   the $\Lie{GL}(n, \mathbb{C})$-complementary  part $\lie{c}$  of the image by the Spencer operator of $\mathrm{T}^* M \otimes \lie{gl}(n, \mathbb{C})$ in  $\Lambda^2 \mathrm{T}^* M \otimes \mathrm{T} M$. 
It  turns  out $\lie{c} = W_1 \oplus  W_2 \oplus  W_3 \oplus W_4 \oplus W_5$, where
$$
W_5 = \{ t \in   \Lambda^2 \mathrm{T}^* M \otimes \mathrm{T} M \, | \, t( X ,  Y ) =  \eta(X)  \eta(t( \zeta  , Y )) \zeta - \eta(Y) \eta(t(\zeta , X)) \zeta \}. 
$$
The  $\Lie{GL}(n, \mathbb{C})$-components of $T_{\lie{c}}$ are given by 
\begin{eqnarray*}
8 T_{W_1} (X,Y) & = &  - 2 N_\varphi (X,Y) +  \eta(X) N_\varphi (\zeta , Y)  -  \eta(Y) N_\varphi (\zeta , X) + 2 d \eta ( \varphi X, \varphi Y))
\\
 2 T_{W_2} (X , Y) & = & -  \eta(X)   N_\varphi (\zeta , Y) +   \eta(Y)   N_\varphi (\zeta , X) 
\\
2 T_{W_3} (X,Y) & = &   ( d \eta ( \varphi^2 X, \varphi^2 Y) + d \eta ( \varphi X, \varphi Y)) \zeta  \\
2 T_{W_4} (X,Y) & = &   ( d \eta ( \varphi^2 X, \varphi^2 Y)   -  d \eta ( \varphi  X, \varphi Y)  ) \zeta   \\
T_{W_5} (X,Y) & = &   (\eta \wedge (\zeta \lrcorner d \eta)) (X,Y) \zeta.
\end{eqnarray*}
The existence of a torsion free $\Lie{GL}(n, \mathbb{C})$-connection is equivalent to the vanishing of the structure tensor. Therefore, in the case of almost contact structure it is equivalent  to $N_\varphi=0$ and $d \eta =0$.  These last conditions are also  equivalent to the integrability of the almost contact structure. We recall that, in general, the integrability of a $\Lie{G}$-structure implies that the corresponding structure tensor is zero. The converse is not true. Only for certain particular cases, as almost complex  structure,  the vanishing of the structure tensor implies  integrability (Newlander-Nirenberg's Theorem). For almost contact structures, this is also the case. 
 }
\end{remark}

Now, in the presence of a compatible metric,  one can relate $N_\varphi$ with $\xi$ by the $\Lie{U}(n)$-map 
$
\xi   \to N(\xi)$, where 
$$
 N(\xi)(X,Y) =  - \varphi ( \xi _X \varphi) Y +   \varphi ( \xi _Y \varphi) X +   ( \xi _{\varphi X} \varphi) Y  - ( \xi _{\varphi Y} \varphi) X = N_\varphi (X,Y).
$$ 
It turns out that $ \mathrm{ker}(N)  =   \mathcal{C}_3 \oplus \mathcal{C}_4 \oplus \mathcal{C}_5 \oplus \mathcal{C}_8 \oplus \mathcal{C}_{10,11} \oplus \mathcal{C}_{12}$, where $\mathcal{C}_{10,11}\cong \lcf \lambda^{2,0} \rcf$ denotes certain diagonal in the space $\mathcal{C}_{10} \oplus \mathcal{C}_{11} \cong \lcf \lambda^{2,0} \rcf \oplus \lcf \lambda^{2,0} \rcf$,  and  
\begin{align*}
&N_{\varphi \, W_1}  =  N(\xi_{(1)} + \xi_{(2)}), 
 &N_{\varphi \, W_2}   =   N(\xi_{(9)} + \xi_{(10)}  + \xi_{(11)}) - 2 ( \xi_{(10)} \eta) \otimes \zeta ,
 \\
  &N_{\varphi \, W_3}  =  N(\xi_{(6)} + \xi_{(7)}) = d \eta_{[ \lambda^{1,1}]} \otimes  \zeta,
  &N_{\varphi \, W_4}  =  - d \eta_{\lcf \lambda^{2,0} \rcf} \otimes  \zeta  = 2( \xi_{(10)} \eta) \otimes \zeta. 
  \end{align*}
Therefore, the remaining information about $\xi$ above mentioned, included in $\xi_{(2)}$ and $\xi_{(9)}$, is located in  $N_{\varphi \, W_1}$ and $N_{\varphi \, W_2}$.

\begin{example}[\bf The hyperbolic space]
{\rm The following example has been already considered in \cite{ChineaGonzalezDavila,GDMC1}.  Let $\mathcal{H} = \{ (x_1, \dots , x_{2n+1}) \in \mathbb{R}^{2n+1} \; | \; x_1 >0\}$ be the $(2n+1)$-dimensional hyperbolic space  with the Riemannian metric
  $$
\textstyle  \langle \cdot , \cdot  \rangle = \frac1{c^2x_1^2} \left( dx_1 \otimes dx_1 + \ldots +   dx_{2n+1} \otimes dx_{2n+1} \right).
  $$
With respect to this metric,   $\{ E_1, \ldots , E_{2n+1}\}$ is an orthonormal frame field, where $
E_i = c x_1 \frac{\partial}{\partial x_i}$,  $i=1, \ldots , 2n+1$.  For the Lie brackets, one has $[E_1, E_j] = cE_j$, $j=2 , \ldots , 2n+1$. The remaining Lie brackets relative to this frame are zero.

The corresponding metrically equivalent coframe is  $\{ e_1, \ldots , e_{2n+1}\}$, where  $e_i = \tfrac{1}{ cx_1} dx_i$. Note that 
$ de_i =  - ce_1  \wedge e_i$,  $i=1, \ldots , 2n+1$.

The almost contact metric structure $(\varphi = \sum_{i,j=1}^{2n+1} \varphi^i_j e_j \otimes E_i, \zeta , \eta, \langle \cdot , \cdot \rangle)$ is considered in \cite{ChineaGonzalezDavila}. The functions $\varphi^i_j$ are constant, $n\geq 2$ and  
$\zeta = \sum_{i=1}^{2n+1}   x_1 k_i \frac{\partial}{\partial x_i} =  \sum_{i=1}^{2n+1} \frac{k_i}{c} E_i$, being $k_i=$ constant  and $k_1^2 + \ldots + k_{2n+1} ^2 = c^2$.
The one-form $\eta$ and the  fundamental form $F$ 
are given by 
$$
\textstyle \eta = \sum_{i=1}^{2n+1}   \frac{k_i}{c} e_i,  \quad F = \sum_{i,j=1}^{2n+1} \varphi^i_j e_i \wedge e_j. 
$$
Then their exterior derivatives are expressed as
$ 
d \eta = - c e_1 \wedge \eta$,  $dF = - 2 c e_1 \wedge F.
$
Hence $d \eta = \xi_\zeta \eta \wedge \eta \in \lcf \lambda^{1,0} \rcf \cong \mathcal{C}_{12}$, where $\xi_\zeta \eta = k_1 \eta - c e_1$, and $dF = 2 (k_1 \eta - c e_1) \wedge F - 2 k_1 \eta  \wedge F \in \lcf \lambda^{1,0} \rcf \wedge F + \mathbb{R} \eta \wedge F \cong \mathcal{C}_{4} \oplus \mathcal{C}_{5}$. Since $dF_{\lcf \lambda^{1,0} \rcf \wedge F} =  
\theta  \wedge F$ and $dF_{\mathbb{R} \eta  \wedge F} = - \frac{d^*\eta}{n} \eta \wedge F$, we have 
$$
\textstyle 
\theta  =    
2 \xi_\zeta \eta, \quad d^* \eta = 4n k_1. 
$$ 

Now, by using the Lie brackets described above, one can check that $N_\varphi (E_i , E_j) = 0$. From all of this,  the structure is of type $\mathcal{C}_4 \oplus \mathcal{C}_5 \oplus \mathcal{C}_{12}$ as it  was shown  in \cite{ChineaGonzalezDavila}.

Note that $ d \theta 
= k_1 
\theta \wedge \eta$, 
 $d\xi_\zeta \eta = k_1 \xi_\zeta \eta \wedge \eta$ and  $d(d^*\eta) = 0$.  From  $d\xi_\zeta \eta = k_1 \xi_\zeta \eta \wedge \eta$,  using Lemma  \ref{comp12}, we deduce  $\nabla^{\Lie{U}(n)}_\zeta \xi_\zeta \zeta =0$. This can be checked by using $\nabla^{\Lie{U}(n)} = \nabla + \xi$ and taking into account that for the Levi Civita connection one has 
  $$
\nabla_{E_i} E_i = c E_1, \qquad \nabla_{E_i} E_1 = - c E_i, \qquad i=2 , \ldots , 2n+1,
$$
being  $\nabla_{E_i} E_j =0$ for the remaining cases. In fact,
$$
\textstyle\nabla^{\Lie{U}(n)}_\zeta \xi_\zeta \zeta = - c \nabla_\zeta E_1 - c \xi_{(12)} E_1 = c \sum_{i=2}^{2n+1} k_i E_i - c ( c \zeta - k_1 E_1) =0.
$$ 
 From these comments, it is also immediate to  check the identities for the components of   $d \theta  
 $ given in Propositions \ref{divergenciaunouno}, \ref{divergenciadoscero} and  \ref{divergencia}.
 
 Particular cases are:

\noindent (i) $k_1 =0$ and $n>1$. The structure is of strict  type $\mathcal{C}_4 \oplus \mathcal{C}_{12}$. The one-forms 
$\theta$ and $\xi_\zeta \eta$ are   closed. In fact, $\xi_\zeta \eta = - d (\ln x_1)$. If we do the conformal change of metric $x_1^2 \langle \cdot , \cdot \rangle$, we obtain the flat cosymplectic structure on $\mathcal{H}$ as an open  set of $\mathbb{R}^{2n+1}$ with the Euclidean metric.

\noindent (ii) $k_1 =0$ and $n=1$. The structure is of strict  type $ \mathcal{C}_{12}$ and $\xi_\zeta \eta$ is    closed.

\noindent (iii) $k_1 =1$.  The structure is of strict  type $\mathcal{C}_5$ with $d^* \eta = 2n$. 

\vspace{2mm}

\noindent {\bf  Some part of $\mathcal{H}$ with another metric}. Now we take the subset $\{ p \in \mathcal{H} \; | \; x_2(p) > 0 \}$ and, on this set, consider $c=1$,   the one-form   $e_{o\,1} = \frac{x_2}{x_1} d x_1$ and the metric
$$
\langle \cdot , \cdot  \rangle_o =  e_{o\,1} \otimes e_{o\,1} + \tfrac{1}{ x_1^2}  \textstyle \sum_{i=2}^{2n+1} dx_i \otimes dx_i. 
$$
An orthonormal frame field is given by $\{ E_{o\,1} =  \frac{ x_1}{x_2} \frac{\partial}{\partial x_1}, E_2, \dots , E_{2n+1} \}$.
\vspace{2mm}

{\it A first  almost contact metric structure}:  Now we consider the almost contact metric structure $(\varphi = \sum_{i,j=2}^{2n+1} \varphi^i_j e_j \otimes E_i, \zeta = E_{o\,1}  , \eta = e_{o\,1}, \langle \cdot , \cdot \rangle_o)$, such that $\varphi^i_j$ are constant and   $n>1$.

The exterior derivative $d \eta$ is given by 
$$
d \eta =  dx_2 \wedge  \tfrac{1}{cx_1} dx_1 = \tfrac{x_1}{x_2} e_2 \wedge \eta \in \lcf \lambda^{1,0} \rcf \wedge \eta \cong \mathcal{C}_{12}.
$$
Hence it is obtained $\xi_\zeta \eta = \frac{x_1}{x_2} e_2 = d( \ln x_2) $ which is closed.

Now taking $d e_i = - \frac{1}{x_2} \eta \wedge e_i$, $i=2, \ldots , 2n+1$, into account, we have 
$$
d F = - \tfrac{2 }{x_2} \eta \wedge F \in \mathbb{R} \eta \wedge F \cong \mathcal{C}_5.
$$
Therefore,  in this case  $d^* \eta =  \frac{2n}{x_2}$ which is not   constant. However,   $d(d^*\eta ) = - d^*\eta \xi_\zeta \eta  $ as  it is expected by  Proposition \ref{divergencia}.  In this case,  $d(d^*\eta)(\zeta) =0$.

For the Lie brackets,  one has $[\zeta, E_2] = \frac{x_1}{x_2} \zeta + \frac{1}{x_2} E_2$,  $[\zeta, E_i] =  \frac{1}{x_2} E_i$, $i=3 , \ldots , 2n+1$. The remaining Lie brackets  are zero.  Taking this into account, it is obtained  $N_\varphi =0$.

From all of this,  we conclude that the almost contact metric structure is of type $\mathcal{C}_5 \oplus \mathcal{C}_{12}$. 

Finally, if we do a conformal change of the metric $ e^{a} \langle \cdot , \cdot \rangle_o$, where $e^a =x_2^{-1}$, we will obtain an almost contact metric structure of type $\mathcal{C}_4 \oplus \mathcal{C}_5$. Denoting the new intrinsic torsion by $\xi_a$ and   the  one-form $\eta_a$ is the one metrically equivalent to the new Reeb vector field, we have $\theta  =  -  2 d(\ln x_2)$, and $d^* \eta_a = 2n$.

\vspace{2mm}

{\it A second almost contact metric structure}:  Now we consider the almost contact metric structure $(\varphi = \sum_{i,j=1}^{2n+1} \varphi^i_j e_j \otimes E_i, \zeta =E_2 , \eta , \langle \cdot , \cdot \rangle_o)$, such that $\varphi^i_j$ are constant and   $n >1$. In this case, we denote $e_{o\,1} =   \frac{x_2}{ x_1} d x_1$, $de_{o\,1} =  - \frac{ x_1}{x_2} \eta \wedge e_{o\,1}$, and $d e_i =- \frac{1}{x_2} e_{o\,1}  \wedge e_i$, $i=3, \ldots , 2n+1$. The exterior derivatives of $\eta$ and $F$ are  given by
$$
d \eta = - \tfrac{1}{x_2} e_{o\,1}  \wedge \eta \in \lcf \lambda^{1,0} \rcf \wedge \eta \cong \mathcal{C}_{12}.
$$
$$
dF  =  - \tfrac{2}{x_2} e_{o\,1} \wedge F   + \tfrac{x_1}{ x_2} \eta \wedge \varphi e_{o\,1} \wedge e_{o\,1}\in \lcf \lambda^{1,0} \rcf \wedge F + \eta \wedge  [\lambda^{1,1} ]   \cong \mathcal{C}_4 \oplus \mathcal{C}_5 \oplus \mathcal{C}_8.
$$

  For the Nijenhuis tensor we obtain
 $$
 N_\varphi = \tfrac{1}{x_2}  e_{o\,1} \wedge \varphi e_{o\,1} \otimes \varphi E_{o\,1} +   \tfrac{1}{x_2} \textstyle \sum_{i=3}^{2n+1}  e_{o\,1} \wedge e_i \otimes E_i
  -  \tfrac{x_1}{x_2} \eta \wedge  e_{o\,1} \otimes  E_{o\,1} 
 +  \tfrac{x_1}{x_2} \eta \wedge \varphi e_{o\,1} \otimes \varphi E_{o\,1}.
 $$ 
Thus in this case $N_\varphi \in N(\mathcal{C}_2) \oplus N(\mathcal{C}_9)$. Therefore, the almost contact structure is of type
$\mathcal{C}_{2} \oplus \mathcal{C}_{4} \oplus 
 \mathcal{C}_{5} \oplus  \mathcal{C}_{8}  \oplus   \mathcal{C}_{9} \oplus  \mathcal{C}_{12}$.
 
 Note that in this case 
$
\theta = 2 \xi_{\zeta} \eta = - \tfrac{(n-1) }{x_2} e_{o\,1}$,   $d^*\eta = -\tfrac{x_1}{x_2}.
$
Hence
$
d \theta =  d \xi_{\zeta} \eta =  0 ,  \; \;  d(d^*\eta) =  d^*\eta  \xi_\zeta \eta   +  (d^*\eta)^2 \eta \;$ and  $\;  d(d^*\eta \; \eta)=0.
$
Note that $d (d^* \eta)(\zeta) =  (d^*\eta)^2$ which is not constant.

Finally, for the conformal change of the metric $ e^{2a} \langle \cdot , \cdot \rangle_o$, where $e^a = x_1x_2^{-\frac1{2n}}$, we will obtain an almost contact metric structure of type $\mathcal{C}_2 \oplus \mathcal{C}_8 \oplus \mathcal{C}_9$. In fact,     $\eta_a$ is closed and 
$$
dF_a = x_2^{-\frac{2n-1}{2n}}  \eta_a \wedge \varphi_a e_{a\,1} \wedge e_{a\,1} - \tfrac{1}{n}   x_2^{-\frac{2n-1}{2n}}  \eta_a \wedge F_a \in \eta \wedge [\lambda_0^{1,1} ] \cong \mathcal{C}_8.
$$
}
\end{example}

\end{document}